\documentclass[10pt]{amsart}

\usepackage{mathtools}
\usepackage{amsthm}
\usepackage{mathrsfs} 
\usepackage{amssymb}
\usepackage[all]{xy}
\usepackage[nobottomtitles*]{titlesec}
\usepackage{titletoc}
\usepackage{bbm}
\usepackage{upgreek}
\usepackage{ marvosym}
\usepackage{multicol}
\usepackage{lscape}
\usepackage{ifluatex}
\usepackage[pdfencoding=auto, psdextra]{hyperref}
\usepackage[noabbrev,nameinlink]{cleveref}
\usepackage{enumerate}
\usepackage{soul}
\usepackage{stmaryrd}

\usepackage{verbatim}
\usepackage{tikz-cd}
\makeatletter
\tikzcdset{
    cong/.style={"\cong" {sloped, description, yshift=0pt,#1}, phantom},
    simeq/.style={"\simeq" {sloped, description, yshift=0pt,#1}, phantom},
    snake/.style={
        out= east, in=west,
        to path={
            \pgfextra{
                \pgfextractx{\pgf@xa}{\pgfpointanchor{\tikztostart}{east}}
                \pgfextractx{\pgf@xb}{\pgfpointanchor{\tikztotarget}{west}}
                \pgfextracty{\pgf@ya}{\pgfpointanchor{\tikztostart}{center}}
                \pgfextracty{\pgf@yb}{\pgfpointanchor{\tikztotarget}{center}}
                \edef\tikzstartx{\the\pgf@xa}
                \edef\tikzendx{\the\pgf@xb}
                \edef\midy{\the\dimexpr0.5\dimexpr\pgf@ya\relax +0.5\dimexpr\pgf@yb\relax}
            }
            to[in=0,out=180,looseness=0.5] (\tikzstartx,\midy)
            -| ([xshift=-2ex]\tikztotarget.west)
            -- (\tikztotarget)}
    }
}
\makeatother
\usepackage{sseq}
\usepackage{tikzcdintertext}
\usepackage{etoolbox}
\usepackage{xpatch}

\newcommand{\cal}{\mathcal}

\crefname{diagram}{diagram}{diagrams}
\creflabelformat{diagram}{#2(#1)#3}
\crefname{sseq}{}{}
\crefname{equation}{}{}
\crefname{equation-b}{Equation}{Equations}
\creflabelformat{sseq}{#2(#1)#3}
\usepackage[normalem]{ulem}
\newcommand{\stkout}[1]{\ifmmode\text{\sout{\ensuremath{#1}}}\else\sout{#1}\fi}

\makeatletter

\pretocmd{\maketitle}{%
    \edef\@title{\unexpanded{\protect\Large}\unexpanded\expandafter{\@title}}%
    \edef\authors{\unexpanded{\protect\normalsize}\unexpanded\expandafter{\authors}}%
}{}{\error}

\def\addlabeltolink#1{\addlabeltolink@#1}
\def\addlabeltolink@#1#2#3#4{#1{#2}{#3}{\thecontentslabel. #4}}

\dottedcontents{section}[0em]{}{2.3em}{1pc}
\dottedcontents{section}[3.8em]{}{2.3em}{1pc}

\titleformat{\section}[block]{\centering\bfseries\Large}{\thetitle. }{0pt}{}

\def\@secnumpunct{. }
\xpatchcmd{\proof}{\topsep6\p@\@plus6\p@\relax}{\topsep0pt\relax}{}{\error}
\makeatother
\usepackage[parfill]{parskip}

\allowdisplaybreaks[1]

\newcommand{\Set}{{\sf \mathcal{S}et}}

\newcommand\Einfty{\mathbb{E}_{\infty}}

\newcommand{\ull}[1]{\underline{#1}}

\newcommand{\mr}[1]{\mathrm{#1}}
\usepackage[nobottomtitles*]{titlesec}
\usepackage{titletoc}

\newcommand{\pic}{{\sf pic}}
\renewcommand{\psi}{\uppsi}

\newcommand{\PP}{\mathbb{P}}

\newcommand{\mJJ}{\mathfrak{J}}

\newcommand{\bo}{{\sf bo}}
\newcommand{\ko}{{\sf ko}}
\newcommand{\ku}{{\sf ku}}
\newcommand{\kr}{{\sf kr}}

\newcommand{\R}{\mr{R}}

\newcommand{\KU}{\mr{KU}}

\newcommand{\RO}{\mr{RO}}

\newcommand*{\rom}[1]{\expandafter\@slowromancap\romannumeral #1@} \makeatother

\hypersetup{%
  bookmarksnumbered=true,%
  bookmarks=true,%
  colorlinks=true,%
  pdfnewwindow=true,%
  pdfstartview=FitBH%
}

\hypersetup{
    colorlinks,
    linkcolor={DarkBrown},
    citecolor={DarkBlue},
    urlcolor={blue!80!black}
}

\def\@url#1{{\tt\def~{\lower3.5pt\hbox{\char'176}}\def\_{\char'137}#1}}

\makeatletter
\let\c@lemma\c@theorem
\makeatother

\newtheorem{thm}[equation]{Theorem}

\newtheorem{lem}[equation]{Lemma}

\newtheorem{prop}[equation]{Proposition}

\theoremstyle{definition}
\newtheorem{defn}[equation]{Definition}

\newtheorem{ex}[equation]{Example}

\newtheorem{rmk}[equation]{Remark}

\newtheorem{notn}[equation]{Notation}

\newtheorem*{thm*}{Theorem}
\newtheorem*{cor*}{Corollary}
\newtheorem*{lem*}{Lemma}
\newtheorem*{prop*}{Proposition}
\newtheorem*{not*}{Notation}
\newtheorem*{guess*}{Guess}

\newtheorem*{defn*}{Definition}
\newtheorem*{ex*}{Example}
\newtheorem*{exs*}{Examples}
\newtheorem*{rmk*}{Remark}
\newtheorem*{claim*}{Claim}
\newtheorem*{exer*}{Exercise}

\numberwithin{equation}{section}
\numberwithin{figure}{section}

\makeatletter
\let\c@lem=\c@thm
\let\c@cor=\c@thm
\let\c@prop=\c@thm
\let\c@lem=\c@thm
\let\c@ex=\c@thm
\let\c@exs=\c@thm
\let\c@obs=\c@thm
\let\c@rmk=\c@thm
\let\c@perthm=\c@thm
\let\c@conjtel=\c@thm
\let\c@exmps=\c@thm
\let\c@rem=\c@thm
\let\c@question=\c@thm
\let\c@warn=\c@thm
\let\c@claim=\c@thm
\let\c@quest=\c@thm
\let\c@notation=\c@thm
\let\c@note=\c@thm
\let\c@conjtel=\c@thm
\let\c@gue=\c@thm
\let\c@goal=\c@thm
\makeatother

\DeclareMathOperator{\Hom}{Hom}

\DeclareMathOperator{\Res}{Res}

\DeclareMathOperator*{\Gcolim}{G-colim}

\DeclareMathOperator{\Sp}{\mathcal{S}p}

\DeclareMathOperator{\Top}{\mathcal{T}op}

\DeclareMathOperator{\gl}{gl_1}

\DeclareMathOperator{\bgl}{bgl_1}

\DeclareMathOperator{\res}{res}
\DeclareMathOperator{\tr}{tr}

\newcommand{\All}{\cal{A}\ell \ell}

\def\makecommands#1#2#3{
    \bgroup
    \def\tempcmdname##1{#1}
    \def\tempcmdbody##1{#2}
    \def\\##1{\expandafter\xdef\csname\tempcmdname{##1}\endcsname{\unexpanded\expandafter{\tempcmdbody{##1}}}}
    #3
    \egroup
}
\def\upperalphabet{\\A\\B\\C\\D\\E\\F\\G\\H\\I\\J\\K\\L\\M\\N\\O\\P\\Q\\R\\S\\T\\U\\V\\W\\X\\Y\\Z}
\def\loweralphabet{\\a\\b\\c\\d\\e\\f\\g\\h\\i\\j\\k\\l\\m\\n\\o\\p\\q\\r\\s\\t\\u\\v\\w\\x\\y\\z}
\def\lowergreekalphabet{\\\alpha\\\beta\\\gamma\\\delta\\\epsilon\\\zeta\\\eta\\\theta\\\kappa\\\lambda\\\mu\\\nu
    \\\xi\\\pi\\\rho\\\sigma\\\tau\\\upsilon\\\psi\\\chi\\\phi\\\omega}

\makecommands{#1mr}{\mathrm{#1}}{\upperalphabet}
\makecommands{#1cal}{\mathcal{#1}}{\upperalphabet}
\makecommands{#1#1}{\mathbb{#1}}{\upperalphabet}
\makecommands{#1bar}{\overline{#1}}{\upperalphabet}
\makecommands{#1twee}{\widetilde{#1}}{\upperalphabet\loweralphabet}
\makecommands{sf#1}{{\sf #1}}{\upperalphabet\loweralphabet}
\makeatletter
\makecommands{\expandafter\@gobble\string#1bar}{\overline{#1}}{\lowergreekalphabet}
\makecommands{\expandafter\@gobble\string#1twee}{\widetilde{#1}}{\lowergreekalphabet}
\makeatother

\newcommand{\FG}{\uFin^\mr{G}}

\newcommand{\uFin}{\ensuremath{\underline{\mathrm{Fin}}_{*}}}

\newcommand{\sma}{\wedge}

\definecolor{limegreen}{rgb}{0.2, 0.8, 0.2}
\definecolor{darkmagenta}{rgb}{0.55, 0.0, 0.55}
\definecolor{lavenderrose}{rgb}{0.91, 0.33, 0.5}
\definecolor{goldenpoppy}{rgb}{0.99, 0.76, 0.0}
\definecolor{seagreen}{rgb}{0.1, 0.4, 0.1}
\definecolor{maroon}{RGB}{128,0,0}
\definecolor{darkviolet}{RGB}{148,0,211}
\definecolor{darkbrown}{rgb}{.55,.3,.2}
\definecolor{snowblue}{rgb}{0.3, 0.4, 0.8}
\definecolor{DarkBlue}{rgb}{.05, 0.30, 0.5} 
\definecolor{DarkBrown}{rgb}{.5, 0.2, 0.2} 
\definecolor{orange}{rgb}{.8, .3, .1} 

\newcommand{\change}[1]{\textcolor{orange}{#1}}

\makeatletter

\def\tikzcdequalsignoffset{0.1em}

\def\findedgesourcetarget#1#2{
    \let\sourcecoordinate\pgfutil@empty
    \ifx\tikzcd@startanchor\pgfutil@empty 
        \def\tempa{\pgfpointanchor{#1}{center}}
    \else
        \edef\tempa{\noexpand\pgfpointanchor{#1}{\expandafter\@gobble\tikzcd@startanchor}} 
        \let\sourcecoordinate\tempa
    \fi
    \ifx\tikzcd@endanchor\pgfutil@empty 
        \def\tempb{\pgfpointshapeborder{#2}{\tempa}}
    \else
        \edef\tempb{\noexpand\pgfpointanchor{#2}{\expandafter\@gobble\tikzcd@endanchor}}
    \fi
    \let\targetcoordinate\tempb
    \ifx\sourcecoordinate\pgfutil@empty%
        \def\sourcecoordinate{\pgfpointshapeborder{#1}{\tempb}}%
    \fi
}

\tikzset{/tikz/commutative diagrams/equal/.style=equals,
    /tikz/commutative diagrams/equals/.style = {
    -,
    to path={\pgfextra{
        \findedgesourcetarget{\tikzcd@ar@start}{\tikzcd@ar@target} 
        \ifx\tikzcd@startanchor\pgfutil@empty
            \def\tikzcd@startanchor{.center}
        \fi
        \ifx\tikzcd@endanchor\pgfutil@empty
            \def\tikzcd@endanchor{.center}
        \fi
        \pgfmathanglebetweenpoints{\pgfpointanchor{\tikzcd@ar@start}{\expandafter\@gobble\tikzcd@startanchor}}{\pgfpointanchor{\tikzcd@ar@target}{\expandafter\@gobble\tikzcd@endanchor}}
        \pgftransformrotate{\pgfmathresult}
        \pgfpathmoveto{\pgfpointadd{\sourcecoordinate}{\pgfpoint{0}{\tikzcdequalsignoffset}}}
        \pgfpathlineto{\pgfpointadd{\targetcoordinate}{\pgfpoint{0}{\tikzcdequalsignoffset}}}
        \pgfpathmoveto{\pgfpointadd{\sourcecoordinate}{\pgfpoint{0}{-\tikzcdequalsignoffset}}}
        \pgfpathlineto{\pgfpointadd{\targetcoordinate}{\pgfpoint{0}{-\tikzcdequalsignoffset}}}
        \pgfusepath{draw}
}}}}

\makeatother

\title{Equivariant orientation of vector bundles over disconnected base spaces \ \ }
\author{Prasit~Bhattacharya}\address{New Mexico State University}\email{prasit@nmsu.edu}
\author{Foling Zou}\address{ Chinese Academy of Sciences}\email{zoufoling@amss.ac.cn}

\thanks{}

\begin{document}

\begin{abstract} In this paper, we view the equivariant orientation theory of equiva	riant vector bundles from the lenses of equivariant Picard spectra. This viewpoint allows us to identify, for a finite group $\mathrm{G}$, a precise condition under which an $\mathrm{R}$-orientation of a $\mathrm{G}$-equivariant vector bundle is encoded by a Thom class. Consequently, we are able to construct a generalization of the first Stiefel--Whitney class of a ``homogeneous" $\mathrm{G}$-equivariant bundle with respect to an $\mathbb{E}_\infty^{\mathrm{G}}$-ring spectrum $\mathrm{R}$. As an application, we show that the $2$-fold direct sum of any homogeneous bundle is $\mathrm{H}\underline{\mathcal{A}}_{\mathrm{G}}$-orientable, where $\underline{\mathcal{A}}_{\mathrm{G}}$ is the Burnside Mackey functor. We notice that  $\mathrm{H}\underline{\mathcal{A}}_{\mathrm{G}}$-orientability is equivalent to $\mathrm{H}\underline{\mathbb{Z}}$-orientability when the order of $\mathrm{G}$ is odd. When the order of $\mathrm{G}$ is even, we show that a $\mathrm{G}$-equivariant analog of the tautological line bundle over $\mathbb{RP}^\infty$ is $\mathrm{H}\underline{\mathbb{Z}}$-orientable but not $\mathrm{H}\underline{\mathcal{A}}_{\mathrm{G}}$-orientable.  
 \end{abstract}

\maketitle	
\tableofcontents
\begin{center}
MSC classes: 55R91, 55R40, 55R50, 19L20, 55S91, 55P91
\end{center}

\section{Introduction} \label{Sec:intro}
  In differential topology, an orientation of a vector bundle is an
  orientation of each fiber which is compatible with the topology of the base space. 
  In homotopy theory, an orientation is expressed as a Thom class in the ordinary cohomology with $\ZZ$-coefficients of its Thom space.  The concept of Thom class 
extends to extraordinary cohomology theories: 
For any ring spectrum $\mr{R}$,  an $\mr{R}$-Thom class of a vector bundle  $\upxi$ is a cohomology class
in the  $\mr{R}$-cohomology of its Thom space 
\[ {\sf u}_{\upxi} \in  [\mr{Th}(\upxi),  \Sigma^{\dim \upxi} \mr{R} ] = \mr{R}^{\dim \upxi}(\mr{Th}(\upxi)) \]
  such that its restriction to the Thom space  over any point is a unit in  the homotopy groups of $\mr{R}$.  This definition led to many interesting ideas and results, which are now foundational in homotopy theory and its application to geometry. 
 
 Equivariant orientation theory was systematically developed in the  eighties and nineties, geometrically in \cite{MayOrient, CMW},
 and cohomologically in \cite{LMMS, CW, Alaska, MayOrient}. Unlike the nonequivariant orientation theory, the geometric theory fails to coincide with the cohomological theory (see \cite[pg 4--5]{CMW}), although a comparison  appears in \cite{MayOrient}.
 One of the challenges of these developments was that many important results,
 such as the equivariant Poincare duality \cite[XVI $\mathsection$9]{Alaska}  were best understood for $\mr{G}$-connected spaces. This is because the equivariant dimension of a vector bundle is not well defined unless the base space is $\mr{G}$-connected. Two different ideas came in as a remedy: Costenoble and Waner \cite{CW} used representations of the fundamental groupoid of the base space, instead of the representation ring of $\mr{G}$, to index  cohomology theories, whereas May \cite{MayOrient} dropped the idea of a single Thom class and advocated  a ``family of Thom classes" instead.
  The difficulty in glueing such a ``family of Thom classes''   
 lies in the fact that, unlike  nonequivariant spaces, a $\mr{G}$-space is not equivalent to the disjoint union of its  $\mr{G}$-connected components (see \Cref{ex:nocoprod}). 
\begin{rmk}[$\mr{G}$-connected components] 
The seminal work of Elmendorf \cite{Esystems}  was extended in \cite{Piacenza}  to show  that $\Top^{\mr{G}}$,  the category   of $\mr{G}$-spaces with $\mr{G}$-equivariant continuous maps,  is Quillen equivalent to the category of
 $\mathcal{O}_{\mr{G}}$-spaces. Given a $\mr{G}$-space $\mr{X}$, we get an $\mathcal{O}_{\mr{G}}$-space
\[ 
\begin{tikzcd}
\Phi(\mr{X}): \cal{O}_{\mr{G}}^{\mr{op}} \rar & \Top,
\end{tikzcd}
\]
that maps $\mr{G}/\mr{H}$ to the $\mr{H}$-fixed points of $\mr{X}$ for any subgroup $\mr{H}$. 
For any $x \in \mr{X}^{\mr{G}}$, by choosing the path component of $x$
in $\mr{X}^{\mr{H}}$ for each $\mr{H} \subset \mr{G}$, one can form  an $\cal{O}_{\mr{G}}$-space. The corresponding $\mr{G}$-spaces  obtained using \cite[Theorem 1]{Esystems} is then a $\mr{G}$-connected component of $\mr{X}$. 
\end{rmk} 
\begin{defn}
  We  say a $\mr{G}$-space $\mr{X}$ is \emph{$\mr{G}$-connected
     when the space of  fixed points $\mr{X}^{\mr{H}}$ has
  exactly one path component for all $\mr{H} \subset \mr{G}$. Otherwise, we say
  $\mr{X}$ is \emph{$\mr{G}$-disconnected} (which includes the case when
    $\mr{X}^{\mr{G}}$ is empty).}
\end{defn}
\begin{ex} \label{ex:nocoprod} Let $\mr{S}^{\upsigma}$ denote the one-point compactification of $\upsigma$, the sign representation of $\mr{C}_2$. Note that the $\mr{C}_2$-fixed point of $\mr{S}^{\upsigma}$ is $\mr{S}^0$ which has two components. Therefore, 
$\mr{S}^{\upsigma}$ has two different $\mr{C}_2$-connected components both equivalent to $\Sigma (\mr{EC}_2)_+$, where $\mr{EC}_2$ is the total space of the universal  $\mr{C}_2$-principal bundle. The space $\mr{S}^{\upsigma}$ is not equivalent to the disjoint union of its $\mr{G}$-connected components as the underlying space of $\mr{S}^{\upsigma}$ is connected. 
 \end{ex}

In this paper, we propose a different approach to equivariant orientation theory,  where  the existence of an  orientation of an equivariant bundle is determined through a two-step obstruction theory.  As a consequence, we are able to make explicit calculations to check whether a given equivariant bundle admits an orientation or not, even when its base space is equivariantly disconnected.

 Taking advantage of the classification theory of equivariant vector bundles \cite{TtDClassify} (also see \cite{MayRmkClassify,FolingClassify}), equivariant $\mr{J}$-homomorphism (see \Cref{rmk:genuineJ}), and  the recent construction of  equivariant Picard spectra \cite{HHKWZ},  we introduce the notion of \emph{homogeneity} of  $\mr{G}$-equivariant vector bundles relative to any $\mathbb{E}_\infty^{\mr{G}}$-ring spectrum $\R$  (see \Cref{defn:relhomogeneous}). In our work,  $\R$-homogeneity is a requirement for an equivariant vector bundle  to be $\R$-orientable, hence should be thought of as the first obstruction. Our first result is an  equivariant generalization of the classical Thom isomorphism theorem where base spaces of  $\mr{G}$-equivariant vector bundles are allowed to be  $\mr{G}$-disconnected.
\begin{thm} \label{main1} Let $\mr{G}$ be a finite group and let $\R$ be an $\mathbb{E}_\infty^{\mr{G}}$-ring spectrum. Suppose that $\upxi$ is an $\mr{R}$-homogeneous $\mr{G}$-equivariant real vector bundle over a $\mr{G}$-space $\mr{B}$,   then the following statements are equivalent: 
\begin{enumerate}
\item  The bundle $\upxi$ is $\mr{R}$-orientable. 
\item There exists an $\mr{R}$-Thom class  \[ {\sf u}_{\upxi} \in [\mr{Th}(\upxi), \cal{I}  ]^{\mr{G}} \] for some invertible $\R$-module $\cal{I}$. 
\item There exists a $\mr{G}$-equivariant weak equivalence  
\begin{equation} \label{eqn:thomiso}
 \begin{tikzcd} 
  \mr{Th}(\upxi) \sma \mr{R} \simeq \mr{B}_+ \sma \cal{I}
\end{tikzcd} 
\end{equation}
for some invertible $\R$-module $\cal{I}$.
\end{enumerate}
\end{thm}
\begin{rmk}[Finiteness assumption on $\mr{G}$] 
 We let $\mr{G}$ to be a finite group in \Cref{main1} and in the rest of the paper to avoid complications of the equivariant infinite loop space theory otherwise (see \cite[$\mathsection$9.2]{MMO}).  Further,  our main application \Cref{thm:gamma_rho_orientation}  makes sense only when $|\mr{G}|$ is finite.
\end{rmk}
 \begin{notn} The real regular representation of any finite group $\mr{G}$ will be denoted by  $\uprho_{\mr{G}}$. When the underlying group $\mr{G}$ is  clear from the context, we will drop the subscript.
\end{notn}

 The notion of $\R$-homogeneity of an equivariant bundle is an extension of the concept of homogeneity (see \Cref{defn:homogeneous}) which already existed in the literature.   Roughly speaking, a $\Gmr$-equivariant vector bundle is considered homogeneous if the fibers are isomorphic as $\Gmr$-representations (see \Cref{rmk:geometry} for a precise statement). This common isomorphism class of $\mr{G}$-representations is then considered to be the equivariant dimension of the bundle.   Costenoble and Waner \cite{CWordinary}  referred to  a $\mr{G}$-equivariant homogeneous bundle   as a ``$\mr{V}$-bundle" 
 where $\mr{V}$ is the $\mr{G}$-representation  encoding the equivariant
 dimension. Note that any $\mr{G}$-equivariant vector bundle over a $\mr{G}$-connected base space is
 homogeneous. However, there are  many  important  and interesting
examples of  homogeneous $\mr{G}$-equivariant vector bundles over
  disconnected base spaces.
 \begin{ex} Let $\mr{V}$ be a finite dimensional real $\mr{G}$-representation. Then the projective space  $\PP(\mr{V})$ (the space of $1$ dimensional subspaces of $\mr{V}$) is a smooth $\mr{G}$-manifold. When $\mr{V} = n \uprho_{\mr{G}}$,  we show (see \Cref{lem:P(V)})  that the tangent bundle of $\PP(\mr{V})$ is a homogeneous bundle of dimension  $n \uprho_{\mr{G}} -1$ (also see \Cref{rmk:TP(V)}). 
 \end{ex}
 \begin{ex} For a real $\Gmr$-representation $\mr{V}$, the tautological
     line bundle $\upgamma_1$ over $\PP(\mr{V})$ is not a homogeneous if
     $\mr{V}$ contains more than one isomorphism classes of one dimensional
     subrepresentations. This is because the $\Gmr$-connected components of
     $\PP(\Vmr)$ are in one-to-one correspondence with isomorphic classes of
     $1$-dimensional subrepresentations of $\mr{V}$, and
     the dimension of $\upgamma_1$ restricted to a component is
     precisely the indexing subrepresentation (as a consequence
     of \Cref{lem:fiber}). A similar conclusion can be made for complex projective space associated to a complex $\Gmr$-representation (also see \Cref{rmk:complexorient}).
 \end{ex}
  \begin{rmk} In our theory, a homogeneous bundle is automatically
      $\Rmr$-homogeneous for all $\EE_\infty^\Gmr$-ring  $\mr{R}$ (compare
      \Cref{defn:homogeneous} and \Cref{defn:relhomogeneous}).
      However,  the converse may not be true in general (see \Cref{rmk:homVSrelhom} and \Cref{rmk:liftrelhom}). 
 \end{rmk}
  We will now describe some universal examples of homogeneous bundles. Let $\Pi$ be a compact Lie group, and let $\mr{E}_{\mr{G}}\Pi$ denote the total space of the universal principal  $\mr{G}$-$\Pi$ bundle (see \eqref{eq:universal} for details). 
 For a finite dimensional real $\Pi \times \mr{G}$  representation $\mr{V}$, let $\uplambda_{\mr{V}}$  be the 
   the  $\mr{G}$-equivariant bundle  whose projection map 
\begin{equation} \label{eqn:generalbundle}
\begin{tikzcd}
\uppi_{\uplambda_{\mr{V}}}: \mr{E}_{\mr{G}} \Pi  \times_{ \Pi} \mr{V} \rar & \mr{E}_{\mr{G}} \Pi  \times_{ \Pi} {\bf 0} := \mr{B}_{\mr{G}} \Pi 
\end{tikzcd}
\end{equation} 
is induced by the terminal map $\mr{V} \twoheadrightarrow {\bf 0}$. The  $\mr{G}$-equivariant bundles of the form $\uplambda_{\mr{V}}$ are important, particularly  when $\Pi$ is the orthogonal group $\mr{O}(n)$,  as the base space $\mr{B}_{\mr{G}} \mr{O}(n)$  classifies $\mr{G}$-equivariant vector bundles \cite{TtDClassify}. However,  one of the difficulties in studying these bundles lies in the fact that the base space of $\uplambda_{\mr{V}}$ is typically $\mr{G}$-disconnected (see \cite{LM} and \Cref{thm:fixBGPi}).
We show that: 
\begin{thm}  \label{thm:Vhomogeneous} Suppose $\mr{V}$ is a $\Pi \times \mr{G}$-representation isomorphic to \[ \mr{V} \cong  n \uprho_{\mr{G}} \otimes \mr{W},\]  where $\mr{W}$ is an arbitrary finite real representation of $\Pi$, and $n \in \NN$. 
Then $\uplambda_{\mr{V}}$ as defined in \Cref{eqn:generalbundle}  is  a homogeneous $\mr{G}$-equivariant vector bundle.  
\end{thm}

\begin{rmk} \label{rmk:C2Steenrod} When $\mr{G} = \mr{C}_2$, $\Pi = \Sigma_2$ and $\mr{V} = \uprho \otimes \uptau$, the tensor product of the regular representation of $\mr{C}_2$ and the sign representation of $\Sigma_2$, then $\uplambda_{\mr{V}}$ is a homogeneous $\mr{C}_2$-equivariant bundle by \Cref{thm:Vhomogeneous}. This fact is essential in the  geometric construction of the $\mr{C}_2$-equivariant Steenrod operations (see \cite[$\mathsection$3.2]{BGL2}). 
 \end{rmk}
\begin{rmk} \label{rmk:C2Steenrod} When $\mr{G} = \mr{C}_2$, $\Pi = \Sigma_2$ and $\mr{V} = \mr{U} \otimes \uptau$, where the $\mr{C}_2$-representation $\mr{U}$ is not in the form of $n\uprho$, it follows from 
\Cref{lem:fiber} that $\uplambda_{\mr{V}}$ is not homogeneous. 
 \end{rmk}

The second and final obstruction to $\R$-orientability is an equivariant analog of the first Stiefel--Whitney class. Classically,  the first Stiefel--Whitney class is  the obstruction to $\mr{H}\mathbb{Z}$-orientability of vector bundles. In \Cref{defn:SW}, we  define for 
an $\R$-homogeneous bundle  $\upxi$,   a cohomology class 
${\sf w}_1^{\R}(\upxi)$, which can be thought of as 
 the first $\mr{G}$-equivariant Stiefel--Whitney class  of $\upxi$ with respect to
  $\R$. 

\begin{thm} \label{thm:w1orient} An $\mr{R}$-homogeneous $\mr{G}$-equivariant bundle $\upxi$ is $\R$-orientable if and only if ${\sf w}_1^{\R}(\upxi) =0$. 
\end{thm}

Similar to the classical  first Stiefel--Whitney class, our generalization also satisfies the additivity formula
\begin{equation} \label{eqn:additivity}
 {\sf w}_1^\R(\upxi_1 \oplus \upxi_2) = {\sf w}_1^\R(\upxi_1) + {\sf w}_1^\R(\upxi_2)  
 \end{equation}
 for a pair of $\R$-homogeneous vector bundles (as usual). 
Using the additivity property, one may  define the \emph{$\R$-orientation order} of an $\mr{R}$-homogeneous vector bundle as the smallest number $n$ for which the $n$-fold direct sum of $\upxi$ is $\R$-orientable.
The additivity is also fundamental in the proof of \Cref{thm:HZtwofold} and \Cref{thm:HAtwofold}. 
 \begin{notn}  Throughout this paper
 \begin{itemize}
 \item $\ull{\mr{A}}$ denotes the constant $\mr{G}$-Tambara functor at the abelian group $\mr{A}$, 
 \item $\ull{\mathcal{A}}_{\mr{G}}$ denotes the Burnside Tambara functor of $\mr{G}$, and 
 \item  $\mr{H}\cal{M}$ denotes  the  Eilenberg--MacLane spectrum corresponding to a Mackey functor $\cal{M}$.
 \end{itemize}
 \end{notn}
 \begin{ex}
 A Tambara functor $\mathcal{T}$ is a  $\mr{G}$-commutative monoid object in the category of Mackey functors, and therefore, 
 the corresponding Eilenberg--MacLane spectrum  $\mr{H} \cal{T}$ is an $\Einfty^{\mr{G}}$-ring spectrum. The first Stiefel--Whitney class of an $\mr{H} \cal{T}$-homogeneous  bundle $\upxi$ over $\mr{B}$ is a class in 
 \[ {\sf w}_1^{\cal{T}}(\upxi):= {\sf w}_1^{\mr{H}\cal{T}}(\upxi)  \in \mr{H}^1(\mr{B}; \cal{T}^{\times}),\] where $\cal{T}^{\times}$ is the unit Mackey functor for $\cal{T}$.
 \end{ex}
 In $\mr{G}$-equivariant homotopy theory, $\mr{H}\underline{\mathbb{Z}}$ is often regarded as a generalization of the integral Eilenberg--MacLane spectrum. In this paper we show that: 
 \begin{thm} \label{thm:HZtwofold} For any $\mr{H}\ull{\mathbb{Z}}$-homogeneous $\mr{G}$-equivariant vector bundle, its $2$-fold direct sum is $\mr{H}\ull{\mathbb{Z}}$-orientable. 
 \end{thm}
 Although $\mr{H}\underline{\mathbb{Z}}$ is a perfectly valid generalization, many
 consider the Burnside Eilenberg--MacLane  spectrum
 $\mr{H}\ull{\cal{A}}_{\mr{G}}$ as a more appropriate one. This is because, just like $\mr{H}\mathbb{Z}$ in the classical case,   $\mr{H}\ull{\cal{A}}_{\mr{G}}$ is the zeroth Postnikov approximation of the sphere spectrum $\mathbb{S}_{\mr{G}}$. Further, $\ull{\cal{A}}_{\mr{G}}$ is the unit in the symmetric monoidal category of Mackey functors (see \cite{Tambara, KMazur}), which is an equivariant generalization of the fact that $\mathbb{Z}$  is the unit for the tensor product of abelian groups. 

Studying $\mr{H}\ull{\cal{A}}_{\mr{G}}$-orientation of equivariant bundles is an important problem as it is a step towards understanding the genuine stable equivalence class (see \Cref{defn:genuinestable}) of  $\mr{H}\underline{\mathbb{Z}}$-orientable vector bundles. 
This is because the unit map for $\mr{H}\underline{\mathbb{Z}}$ factors through that of $\mr{H}\ull{\cal{A}}_{\mr{G}}$
\[ 
\begin{tikzcd}
\mathbb{S}_{\mr{G}} \rar & \mr{H} \ull{\cal{A}}_{\mr{G}} \rar & \mr{H} \ull{\mathbb{Z}}.
\end{tikzcd}
\]
 Moreover, $\mr{H}\ull{\cal{A}}_{\mr{G}}$-orientation of an equivariant vector bundle has a precise geometric interpretation as explained in  \cite{MayOrient}, because of which many important results in the subject, such as equivariant Poincare duality \cite[$\mathsection$1.12.2]{CWordinary}, are discussed in terms of $\mr{H}\ull{\cal{A}}_{\mr{G}}$-cohomology. 
 We show: 
\begin{thm} \label{thm:HAtwofold} For any $\mr{H} \ull{\cal{A}}_{\mr{G}}$-homogeneous $\mr{G}$-equivariant vector bundle, its $2$-fold direct sum is  $\mr{H} \ull{\cal{A}}_{\mr{G}}$-orientable. 
\end{thm}
When $|\mr{G}|$ is odd, the unit presheaves $\ull{\ZZ}^{\times}$ and $\ull{\cal{A}}_{\mr{G}}^{\times}$ are equal, and therefore, 
an $\mr{H}\underline{\mathbb{Z}}$-orientation is sufficient to obtain an $\mr{H}\ull{\cal{A}}_{\mr{G}}$-orientation:
\begin{thm} \label{thm:AorientGodd} Let  $\mr{G}$ be a finite group of odd order. Then \[ {\sf w}_1^{\ull{\cal{A}}_{\mr{G}}} (\upxi) = {\sf w}_1^{\ull{\mathbb{Z}}} (\upxi)\] for any $\mr{H}\ull{\cal{A}}_{\mr{G}}$-homogeneous $\mr{G}$-equivariant bundle $\upxi$. 
\end{thm}
The above theorem does not hold when $\mr{G}$ has an even order because the kernel of the  map 
\[ 
\begin{tikzcd}
\ull{\cal{A}}_{\mr{G}}^{\times} \rar & \ull{\mathbb{Z}}^{\times}
\end{tikzcd}
\]
of $\mr{G}$-Mackey functors is nonzero.  We use this fact to define a set
$\mathfrak{G}_{1}(-)$ (see \Cref{defn:ghosts}),  and refer to its elements as
``ghosts\footnote{The term ghost is used  to indicate that the set $\mathfrak{G}_1(\upxi)$
  is not defined unless ${\sf w}_1^{\ull{\mathbb{Z}}}(\upxi)$ is dead  (i.e. ${\sf
    w}_1^{\ull{\mathbb{Z}}}(\upxi) =0$).} of the first Stiefel--Whitney class". These elements can be regarded as  obstructions to $\mr{H}\ull{\cal{A}}_{\mr{G}}$-orientability for 
 $\mr{H} \ull{\mathbb{Z}}$-orientable vector bundles:
\begin{thm}  \label{thm:ghost} An $\mr{H}\ull{\cal{A}}_{\mr{G}}$-homogeneous $\mr{H}\ull{\mathbb{Z}}$-orientable vector bundle $\upxi$ is $\mr{H}\ull{\cal{A}}_{\mr{G}}$-orientable if and only if $0 \in \mathfrak{G}_1(\upxi) $.
\end{thm} 
 We show that  the trivial element always appears as a ghost when $\mr{G}$ acts freely on the base space of a bundle.  Consequently, we  have the following result. 
\begin{thm} \label{freeVSghost}  
Let $\upxi$ be a $\mr{H}\ull{\cal{A}}_{\mr{G}}$-homogeneous bundle such that $\mr{G}$ acts freely on 
its base space. Then $\upxi$  is $\mr{H}\ull{\cal{A}}_{\mr{G}}$-orientable if and only if it is $\mr{H}\ull{\mathbb{Z}}$-orientable. 
\end{thm}
\medskip
In 1966, M.F.  Atiyah  \cite{AtiyahKReal} introduced a subclass of $\mr{C}_2$-equivariant vector bundles whose underlying nonequivariant bundle is a complex vector bundle and the action of $\mr{C}_2$ on the fibers is compatible with the complex conjugation action on $\mathbb{C}$. These bundles are referred to as Atiyah Real bundles. His work implies that the classifying space of Atiyah Real bundles is a ``genuine" $\mr{C}_2$-equivariant infinite loop space, and hence deloops to a  genuine $\mr{C}_2$-spectrum $\kr$  called the connective Atiyah Real $\mr{K}$-theory. 

The tautological Atiyah Real line bundle  $\hat{\upgamma}$ is the
  tautological complex line bundle over the  infinite complex projective space
  equipped with the complex conjugation action of $\mr{C}_2$.
 The Thom isomorphism of  \cite[Theorem 2.4]{AtiyahKReal}  implies that the bundle $\hat{\upgamma}$ is $\kr$-orientable. This automatically implies that $\hat{\upgamma}$ is orientable with respect to $\mr{H}\ull{\mathbb{Z}}$,  the zeroth Postnikov approximation of $\kr$. We  show  that: 
\begin{thm} \label{thm:ARnotorientable}The  tautological Atiyah Real bundle $\hat{\upgamma}$  is not $\mr{H} \ull{\cal{A}}_{\mr{C}_2}$-orientable. 
\end{thm}
Let $\upgamma_{\uprho}$ denote the $\uprho$-dimensional bundle over $\mr{B}_{\mr{G}}\Sigma_2$ obtained by setting  $\Pi = \Sigma_2$ and $\mr{V} = \uprho \otimes \uptau$ in \eqref{eqn:generalbundle}. In some sense, $\upgamma_{\uprho}$
is  the $\mr{G}$-equivariant analog of the real tautological line bundle, therefore fundamental  in the study of  equivariant homotopy theory. In the forthcoming work \cite{BZZ}, we use 
$\upgamma_{\uprho}$  to construct $\mr{G}$-equivariant Steenrod operations  extending \cite[$\mathsection$3]{BGL2}. In this paper, we determine its orientability with respect to 
$\mr{H}\ull{\ZZ}$ as well as $\mr{H}\ull{\cal{A}}_{\mr{G}}$. 

 When $\mr{G} = \mr{C}_2$, there exists a map of $\mr{C}_2$-equivariant vector bundle (see \cite[Remark 3.20]{BGL2})
\begin{equation} \label{C2bundleMap} 
\begin{tikzcd}
\upgamma_{\uprho} \rar & \hat{\upgamma}, 
\end{tikzcd}
\end{equation}
thus an $\mr{H}\ull{\mathbb{Z}}$-orientation of $\hat{\upgamma}$ leads to an  $\mr{H}\ull{\mathbb{Z}}$-orientation of $\upgamma_{\uprho}$. However, 
$\upgamma_{\uprho}$ is not orientable with respect to $\mr{H} \ull{\cal{A}}_{\mr{C}_2}$. We end the paper by generalizing this result to all finite groups of even order using \Cref{SWcompare} which gives an explicit formula for computing the first Stiefel-Whitney class of an induced bundle (as defined in \Cref{Indbundle}):
 
\begin{thm} \label{thm:gamma_rho_orientation} If $\mr{G}$ is a finite group of
  even order, then $\upgamma_{\uprho}$ is $\mr{H}\ull{\mathbb{Z}}$-orientable,
    but not $\mr{H}\ull{\cal{A}}_{\mr{G}}$-orientable.
    \end{thm}

\begin{notn}  This paper relies heavily on equivariant Picard spectra whose construction, unfortunately,  does not appear in the literature. The authors find the modern infinity apparatus convenient to describe their construction in short, and  therefore,  throughout this paper 
\begin{itemize}
\item   $\Top^{\mr{G}}$  is the $\infty$-category associated with simplicial model category of topological $\mr{G}$-spaces, 
\item $\Top^{\mr{G}}_*$  is the $\infty$-category of pointed topological $\mr{G}$-spaces, and 
\item  $\Sp^{\mr{G}}$ is the stable $\infty$-category of genuine
  $\mr{G}$-spectra (see \cite[p2]{Expose2}).
\end{itemize} 
 If  the construction of equivariant Picard spectra is taken  as a black box,
  the results in this paper, and their proofs, are independent of the model of $\Sp^{\mr{G}}$. They 
 rely on the loop-suspension adjunction 
\begin{equation}  \label{loop-suspension}
\begin{tikzcd}
\Sigma^{\infty}_{\mr{G}}:{\bf Ho}(\Top^{\mr{G}}_*) \rar[shift right] &  {\bf Ho}(\Sp^{\mr{G}}) : \Omega^{\infty}_{\mr{G}} \lar[shift right],
\end{tikzcd}
\end{equation}
which always exists between homotopy categories.
\end{notn}
\begin{notn}
We will use $ [- , -]^{\mr{G}} $ to denote the the homotopy class of $\mr{G}$-equivariant maps (as usual) in  
 $\Top^{\mr{G}}_*$  as well as $\Sp^{\mr{G}}$.
\end{notn}

\subsection*{Acknowledgement.} 
This paper has  greatly benefitted  from many insightful exchanges with J.P. May
on the topic of equivariant orientation theory.  We have also benefitted from
our conversations with Vigleik Angelveit, Samik Basu, Agnes Beaudry, Hood Chatham, Bert
Guillou, John Greenlees, Christy Hazel, Mike Hill, Ang Li,   Mike Mandell, Clover May, and Doug Ravenel. We are grateful to the anonymous referee for several useful comments, and  for suggesting a discussion on complex orientations of equivariant cohomology theories (see \Cref{rmk:complexorient}). 

The first author would also like to thank Tata Institute of Fundamental Research at Mumbai for their hospitality during the period when a part of this work was done. 

The research in this paper is supported by the National Science Foundation through grant DMS-2305016.

\subsection*{Organization of the paper}
In \Cref{Sec:EOtheory}, first we discuss the concept of homogeneity and prove \Cref{thm:Vhomogeneous}. This theorem provides examples of homogeneous bundles as it gives a sufficient condition that guarantees homogeneity of equivariant bundle  $ \uppi_{\uplambda_{\mr{V}}}$ of \eqref{eqn:generalbundle}. In \Cref{subsec:relhomogeneity}, we introduce the concept of relative homogeneity, and in \Cref{subsec:orient-thom}, we prove \Cref{main1} which is an equivariant generalization of the classical Thom isomorphism theorem. 

In \Cref{sec:SW}, we define the equivariant generalization of the first Stiefel--Whitney class, and use it to determine orientability of equivariant vector bundles with respect to  $\mr{H}\ull{\ZZ}$ and $\mr{H}\ull{\cal{A}}_{\mr{G}}$. In the process, we prove \Cref{thm:HZtwofold}, \Cref{thm:HAtwofold}, \Cref{thm:AorientGodd}, \Cref{thm:ghost}, \Cref{freeVSghost}, \Cref{thm:ARnotorientable}, and \Cref{thm:gamma_rho_orientation}. 

In  \Cref{Appendix:unit}, we  recall some general formulas needed in identifying the Mackey functor structure of the units of the Burnside Tambara functor $\ull{\cal{A}}_{\mr{G}}^{\times}$.  We also record them explicitly  for a few familiar abelian $2$-groups anticipating future applications.

\section{Equivariant Orientation theory} \label{Sec:EOtheory}

A $\mr{G}$-equivariant vector bundle $\upxi$ over a base space $\mr{B}$ consists of a map $\uppi_{\upxi}: \mr{E} \longrightarrow \mr{B}$ in $\Top_{\mr{G}}$ such that the fiber $\mr{E}_b$ over $b \in \mr{B}$ is a real vector space, and the fiberwise action of  ${\sf g} \in \mr{G}$ 
\[ 
\begin{tikzcd}
{\sf g}: \mr{E}_{b} \rar & \mr{E}_{{\sf g} \cdot  b}
\end{tikzcd}
\]
is a linear isomorphism. 

\begin{defn} \label{defn:genuinestable}
Two $\mr{G}$-equivariant vector bundles, $\upxi_1$ and $\upxi_2$,  are
genuinely stably equivalent  if 
there exists some $\mr{G}$-representation $\mr{V}$ such that 
   \[ \upxi_1 \oplus \upepsilon_{\mr{V}} \cong  \upxi_2 \oplus \upepsilon_{\mr{V}},\]
where $\upepsilon_{\mr{V}}$ is the trivial bundle $\mr{B} \times \mr{V}$ over $\mr{B}$.  
\end{defn}
Genuine stable equivalence classes of vector bundles are classified by the infinite loop space of the genuine $\mr{G}$-equivariant real $\mr{K}$-theory 
\begin{equation} \label{repnkoG} 
\{ \text{$\mr{G}$-equivariant vector bundle over $\mr{B}$} \} /(\sim_{\mr{G}}) \cong [\mr{B}_+,\Omega^{\infty}_{\mr{G}} \ko_{\mr{G}}]^{\mr{G}}, 
\end{equation}
where $\sim_{\mr{G}}$ stands for genuine stable equivalence \cite[pg 134]{SegalequivK}. 

\begin{notn}  \label{notn:classify}
Let  ${\sf f}_{\upxi}: \mr{B} \longrightarrow \Omega_{\mr{G}}^{\infty}\ko_{\mr{G}}$ denote the
classifying map  (in the homotopy category of $\Top^{\mr{G}}$) for a
$\mr{G}$-equivariant bundle $\upxi$ over $\mr{B}$. We will use the same notation
to  denote the corresponding map 
$
\begin{tikzcd}
  \Sigma^{\infty}_{\mr{G}}\mr{B}_+ \rar & \ko_{\mr{G}}
\end{tikzcd}
$ 
of spectra obtained using the  adjunction \eqref{loop-suspension}.
\end{notn}

When $\mr{G}$ is the trivial group, the map (in $\Set$)
\begin{equation} \label{eqn:dimconst}
\begin{tikzcd}
 \pi_0({\sf f}_{\upxi}) : \{ \text{path components of $\mr{B}$}\} \rar & \ZZ 
 \end{tikzcd}
 \end{equation}
assigns each path component the dimension of $\upxi$ over it. 
In order to identify an orientation  using a single Thom class, we require $\pi_0({\sf f}_{\upxi})$ to be  a constant map. 

 We will now introduce the notion of homogeneity for $\mr{G}$-equivariant vector bundles,  which  is essentially a equivariant generalization of   ``${\sf f}_{\upxi}$ inducing a constant map on the zeroth homotopy''.

\subsection{Homogeneity of vector bundles}  \label{subsec:homogeneity} \ 

When $\mr{G}$ is a  nontrivial group,  there are a variety of morphisms in
$\cal{O}_{\mr{G}}$-$\Set$ which can be regarded as constant maps. In order to
describe them, we first recall that a set $\cal{F}$ consisting of subgroups of
$\mr{G}$ is called a \emph{family} if it is closed under inclusions up to
conjugations. For a family $\cal{F}$, define the $\cal{O}_{\mr{G}}$-$\Set$ 
\[ 
{\sf c}_{\cal{F}}(\mr{G}/\mr{H}) = \left\lbrace 
\begin{array}{ccccc}
\ast  & \text{if $\mr{H} \in \cal{F}$} \\
\emptyset & \text{otherwise.}
\end{array}
\right.
\]
\begin{defn}
An $\cal{F}$-point of an  $ \cal{O}_{\mr{G}}$-$\Set$  $\mr{X}$ is an $\cal{O}_{\mr{G}}$-map 
\[
\begin{tikzcd} 
\mr{p}: {\sf c}_{\cal{F}} \rar & \mr{X}.
\end{tikzcd} 
\]
 Thus, an $\cal{F}$-point $\mr{p}$ consists of  a collection of points $\mr{p}_{\mr{H}}\in \mr{X}(\mr{G}/\mr{H}) $ for  each $\mr{H} \in \cal{F}$, such that $ \mr{X}(\alpha)(\mr{p}_{\mr{H}_1}) =\mr{p}_{\mr{H}_2}$ whenever $\mr{H}_1, \mr{H}_2 \in \cal{F}$ and $\alpha:\mr{G}/\mr{H}_2 \to \mr{G}/\mr{H}_1$ is a morphism in $\cal{O}_{\mr{G}}$. 
\end{defn}
\begin{notn}
 Let $\All_{\mr{G}}$ denote the family consisting of all subgroups of $\mr{G}$.  
 \end{notn}
 \begin{notn} For a $\mr{G}$-space $\mr{B}$,  we let 
 \[ 
 \begin{tikzcd}
 \ull{\pi}_0(\mr{B}): \mathcal{O}_{\mr{G}}^{\mr{op}} \rar & \Set 
 \end{tikzcd}
 \]
  be the $\mathcal{O}_{\mr{G}}$-set which assigns $\mr{G}/\mr{H}$ the set of path components of $\mr{B}^{\mr{H}}$ for any subgroup $\mr{H}$ of $\mr{G}$. We forget the basepoint if $\mr{B}$ has one.
 \end{notn}
 \begin{rmk} For a $\mr{G}$-spectrum $\mr{R} \in \Sp^{\mr{G}}$, the $\mathcal{O}_{\mr{G}}$-set
   $\ull{\pi}_0(\Omega^{\infty}_{\mr{G}} \R)$ has a Mackey functor structure as it is isomorphic to   the zeroth homotopy group of $\R$ given by 
  \[ \ull{\pi}_0(\R)(\mr{G}/\mr{H}):= [\mathbb{S}_{\mr{G}} \sma (\mr{G}/\mr{H})_+, \R]^{\mr{G}},\]
  for any subgroup $\mr{H}$ of $\mr{G}$. 
 \end{rmk}
\begin{defn}  \label{defn:homogeneous} We say a $\mr{G}$-vector bundle  $\upxi$ is  \emph{homogeneous}  if $\underline{\pi}_0({\sf f}_{\upxi})$ admits a factorization
\begin{equation} \label{eqn:factor}
 \begin{tikzcd}
&& {\sf c}_{\All_{\mr{G}}} \dar["\mr{p}"] \\
\underline{\pi}_0(\mr{B}) \ar[rru, dashed] \ar[rr,"\ull{\pi}_0 ({\sf f}_{\upxi})"'] &&\underline{\pi}_0( \Omega^{\infty}_{\mr{G}} \ko_{\mr{G}}),
\end{tikzcd}
\end{equation}
through   an $\All_{\mr{G}}$-point $\mr{p}$ of $\underline{\pi}_0\ko_{\mr{G}}$. We then call $\mr{p}$ a \emph{coordinate of homogeneity} of $\upxi$. 
\end{defn}

By  definition (see  \eqref{repnkoG}), the zeroth homotopy of $\ko_{\mr{G}}$   
is the real representation Mackey functor
$\underline{\RO}_{\mr{G}}$ given by  \[ 
\underline{\RO}_{\mr{G}} (\mr{G}/\mr{H}) \cong \RO(\mr{H})
\]
for any subgroup $\mr{H} \subset \mr{G}$. Thus, an $\All_{\mr{G}}$-point $\mr{p}$ of $\underline{\pi}_0(\Omega^{\infty}_{\mr{G}}\ko_{\mr{G}})$ is determined by $\mr{p}_{\mr{G}} \in \RO(\mr{G})$ as $\mr{p}_{\mr{H}}$ is simply the restriction of $\mr{p}_\mr{G}$ to the subgroup $\mr{H}$ of $ \mr{G}$.  Therefore, we declare  $\mr{p}_{\mr{G}}$ as the \emph{dimension}
\[ 
 \dim (\upxi, \mr{p}) := \mr{p}_{\mr{G}}
\]
 of the pair $(\upxi, \mr{p})$.

\begin{rmk}[Geometric interpretation] \label{rmk:geometry} \label{rmk:GIhomogeneity}
  For a homogeneous $\mr{G}$-equivariant vector bundle $\upxi$ whose coordinate of homogeneity is $\mr{p}$, the fiber over any point in $\mr{B}^{\mr{H}}$ is an $\mr{H}$-representation representing the class  $\mr{p}_{\mr{H}} \in \RO(\mr{H})$ (also see \cite[Example 1.1.3]{CWordinary}). 
  \end{rmk}
 The coordinate of homogeneity of a $\mr{G}$-equivariant bundle may not be unique.  Such situations may arise, when the $\mr{H}$-fixed points of the base space is empty for some subgroup $\mr{H} \subset \mr{G}$ as demonstrated by the following example. 
\begin{ex}  Let $\upxi$ denote the $\mr{C}_2$-equivariant bundle  
\[ 
\begin{tikzcd}
\uppi_{\upxi}: 
\mr{S}(2 \upsigma) \times \upsigma \rar &
\mr{S}(2 \upsigma),
\end{tikzcd}
\] 
where $\upsigma$ is the sign representation of $\mr{C}_2$. Then  $\underline{\pi}_0({\sf f}_{\upxi})$ induces a map of $\cal{O}_{\mr{C}_2}$-$\Set$ given by the diagram  
\[
\begin{tikzcd}
\emptyset \rar \dar & \RO(\mr{C}_2) \dar["\res"]  \\
\{ \ast \} \rar & \ZZ
\end{tikzcd}
 \]
 that maps $\ast$ to  $1$. Thus, there are infinitely many coordinates of homogeneity.
\end{ex}

\subsection{Examples of homogeneous bundles}    \

Products and pullbacks of homogeneous bundles are homogeneous, and therefore, the interest lies in knowing which among the universal examples of  $\mr{G}$-equivariant bundles defined in \eqref{eqn:generalbundle} are homogeneous. 
We begin by recalling a result of  Lashof and May  \cite[Theorem 10]{LM} which identifies all $\mr{G}$-connected components of the base space of these bundles. We also sketch a proof following \cite{LM} as the details are necessary in the  proof of \Cref{thm:Vhomogeneous}. 

\begin{thm}[Lashof-May] \label{thm:fixBGPi} For a subgroup $\mr{H}$ of $\mr{G}$, the $\mr{H}$-fixed
  points of $\mr{B}_{\mr{G}} \Pi$  is equivalent to 
\begin{equation} \label{eqn:fixequiv}
 (\mr{B}_{\mr{G}} \Pi)^{\mr{H}} \simeq \bigsqcup_{[\uptheta] \in \Theta} \mr{B}\cal{Z} (\uptheta)
 \end{equation}
 where 
 \begin{itemize}
 \item  $\Theta$ is the set of $\Pi$-conjugacy classes of group homomorphism $\uptheta: \mr{H} \to \Pi$, 
 \item  $\cal{Z} (\uptheta)$ is the subgroup  of $\Pi$ consisting of elements that commute with the image
   of $\uptheta$, which, up to conjugation, is independent of the choice of
   the representative $\uptheta$.
 \end{itemize}
 \end{thm}
\begin{proof}[Sketch proof]
Let $\Lambda_{\uptheta} := (\uptheta(h),h) $ denote the graph subgroup of the group homomorphism $\uptheta: \mr{H} \to \Pi$. The universal property of $\mr{E}_{\mr{G}}\Pi$ implies 
\begin{equation}
\label{eq:universal}
(\mr{E}_{\mr{G}}\Pi)^{\Lambda} \simeq
 \begin{cases}
   * & \Lambda = \Lambda_{\uptheta} \text{ for some } \mr{H},\uptheta\\
   \emptyset & \text{ o.w. }
 \end{cases}
\end{equation}
Then
 \[ \uppi: \mr{E}_{\mr{G}}\Pi \longrightarrow  \mr{B}_{\mr{G}}\Pi\] 
 is the universal $\mr{G}$-equivariant principal $\Pi$ bundle.  
The fiber $\uppi^{-1}(b)$ is a $\Pi \times \mr{H}$-space for each point $b \in  (\mr{B}_{\mr{G}}\Pi)^{\mr{H}}$.

 Notice that for each point $z \in \uppi^{-1}(b)$, the isotropy subgroup is precisely $\Lambda_{\uptheta}$ for some 
\begin{equation} \label{eqn:assohomo}
\begin{tikzcd}
\uptheta: \mr{H} \rar &  \Pi.
\end{tikzcd}  
\end{equation}
It is easy to check that the $\Pi$-conjugacy  class  $[\uptheta]$ is independent of the choice of $z$, and therefore, 
  is constant on each of the connected components of $(\mr{B}_{\mr{G}}\Pi)^{\mr{H}}$. Consequently, we get a map 
  \begin{equation} \label{eqn:alpha}
 \begin{tikzcd} 
  \upalpha:\pi_0^{\mr{H}}(\mr{B}_{\mr{G}}\Pi) \rar & \Theta, 
  \end{tikzcd}
\end{equation}
which is a bijection because of \eqref{eq:universal}. 

Let $\mr{B}$ be a connected component of $(\mr{B}_{\mr{G}}\Pi)^{\mr{H}}$, and let $\uptheta \in \upalpha([\mr{B}])$. Using \eqref{eq:universal}, one concludes  that  $(\uppi^{-1}(\mr{B}))^{\Lambda_{\uptheta}}$ is a contractible space with a free action of $\cal{Z}(\uptheta)$. Therefore, $(\uppi^{-1}(\mr{B}))^{\Lambda_{\uptheta}} \longrightarrow \mr{B}$ is a universal $\cal{Z}(\uptheta)$-bundle and $\mr{B} \simeq \mr{B}\cal{Z}(\uptheta)$. 
  \end{proof}
 
 Given a $(\Pi \times \mr{G})$-space  $\mr{F}$, one can construct a $\mr{G}$-equivariant fiber bundle $\uplambda_{\mr{F}}$ whose projection map 
 \[ 
 \begin{tikzcd}
 \uppi_{\uplambda_{\mr{F}}}: \mr{E}_{\mr{G}}\Pi \times_{\Pi} \mr{F} \rar & \mr{B}_{\mr{G}} \Pi
 \end{tikzcd}
 \]
 is induced by the collapsing map from $\mr{F}$ to a point. 

\begin{lem} \label{lem:fiber} If 
$
\begin{tikzcd}
{\sf a}_{\mr{F}} : \Pi \times \mr{G} \rar & \mr{Aut}(\mr{F})
\end{tikzcd}
$
 is the action map on $\mr{F}$, $\mr{H}$ is a subgroup of $\mr{G}$ and $b \in (\mr{B}_{\mr{G}} \Pi)^{\mr{H}}$, then the action of $\mr{H}$ on $\mr{F}' := \uppi_{\uplambda_{\mr{F}}}^{-1}(b) \cong \mr{F}$ is given by the composite 
\begin{equation} \label{eq:modifyaction}
\begin{tikzcd}
{\sf a}_{\mr{F}'}:\mr{H} \ar[rr,"\text{$(\uptheta, \iota)$}"] && \Pi \times \mr{G} \rar["{\sf a}_{\mr{F}}"] & \mr{Aut}(\mr{F}) \rar[cong] & \mr{Aut}(\mr{F}'),
\end{tikzcd}
\end{equation}
where $\iota: \mr{H} \hookrightarrow \mr{G}$ is the inclusion of $\mr{H}$ into $\mr{G}$ and $\uptheta \in \upalpha(\text{\emph{component of $b$}})$, for $\upalpha$ defined in \eqref{eqn:alpha}.
\end{lem}

\begin{proof}[\textbf{Proof}] Fix $z \in \uppi^{-1}(b)$, and let $\uptheta$ be the corresponding group homomorphism  as in  \eqref{eqn:assohomo}. It follows from the construction of $\uptheta$ that 
\[ h \cdot z = \uptheta(h)^{-1} \cdot z \]
for all $h \in \mr{H}$. 
Let ${\sf r}$ denote the composite  
\[
\begin{tikzcd}
{\sf r} :\mr{F}' := \uppi_{\uplambda_{\mr{F}}}^{-1}(b) \rar["\cong", "{\sf r}_1"'] &\uppi^{-1}(b) \times_{\Pi} \mr{F}  \rar["\cong", "{\sf r}_2"'] & \mr{F} 
\end{tikzcd}
\] 
and  set the last isomorphism in \eqref{eq:modifyaction} as conjugation by ${\sf r} $. 
Then the result follows from the observation that 
\begin{eqnarray*}
h \cdot x &=& {\sf r}^{-1} (h \cdot {\sf r}( x))\\
 &=& {\sf r}^{-1}({\sf r}_2(h \cdot {\sf r}_1(x)  ))  \\
 &=& {\sf r}^{-1}({\sf r}_2(  h \cdot [z,{\sf r}(x)]  ))  \\
  &=& {\sf r}^{-1}({\sf r}_2( [\uptheta(h)^{-1} \cdot z, h \cdot {\sf r}(x)] ) ) \\
  &=& {\sf r}^{-1}({\sf r}_2([z, \uptheta(h)h \cdot {\sf r}(x)  ])) \\
  &=& {\sf r}^{-1} (\uptheta(h)h \cdot {\sf r}(x) ) 
\end{eqnarray*}
for all $h \in \mr{H}$.
\end{proof}
 
 \bigskip 
\begin{proof}[\textbf{Proof of} \Cref{thm:Vhomogeneous}]
Let  $\mr{F}= \mr{U} \otimes \mr{W}$, where
$\mr{U}$ and $\mr{W}$   are real representations of
$\mr{G}$ and  $\Pi$  respectively.  Let $b \in (\mr{B}_{\mr{G}} \Pi)^{\mr{H}}$, and let 
\[ 
\begin{tikzcd}
\uptheta: \mr{H} \rar & \Pi
\end{tikzcd}
\]
 be a group homomorphism such that $[\uptheta] = \upalpha(\text{component of $b$})$. Then it follows from
\Cref{lem:fiber} that 
\begin{equation} \label{eqn:repovercomp}
\uppi_{\uplambda_{\mr{F}}}^{-1}(b) \cong  \uptheta^*(\mr{W}) \otimes  \Res_{\mr{H}}^{\mr{G}} \mr{V}
\end{equation}
as an $\mr{H}$-representation. 

 The regular representation  $\uprho$ of any compact Lie group group  has the property that 
\[ \uprho \otimes \mr{W} = \bigoplus_{\dim \mr{W}} \uprho  \]
for a finite dimensional representation $\mr{W}$. Therefore, when $\mr{V}$ is a direct sum of regular representations of $\mr{G}$ then $\Res_{\mr{H}}^{\mr{G}} \mr{V}$ is a direct sum of regular representations of $\mr{H}$ and 
$\uppi_{\uplambda_{\mr{F}}}^{-1}(b)$ is independent of the choice of $b$ for
all $\mr{H} \subset \mr{G}$. Hence, the result  (also see \Cref{rmk:geometry}). 
\end{proof}

Notice that  \Cref{thm:Vhomogeneous} produces examples of $\Gmr$-equivariant vector bundles whose dimensions are integral multiples of the regular representation. We will now demonstrate examples of homogeneous bundles whose dimensions are not multiples of $\uprho_\Gmr$.

Let $\PP(\mr{V})$ denote the projectification of the finite $\mr{G}$-representation. Since $\PP(\mr{V})$ is a smooth $\mr{G}$-manifold, we may ask if the tangent bundle $\mr{T}\PP(\mr{V})$ is homogeneous.  
\begin{lem} \label{lem:P(V)}
  Suppose $\mr{V} = n \uprho_{\Gmr}$ for  positive integer $n$.  Then the tangent bundle  of  $\PP(\mr{V})$ is a  homogeneous bundle of dimension $n \uprho_{\mr{G}} -1$. 
\end{lem}
\begin{proof} Suppose $\mr{H}$ is a subgroup of  $\mr{G}$ and $\ell \in\PP(\Vmr)^\Hmr$.  Then the pullback  $\Tmr \PP(\mr{V})|_\ell$ 
\[ 
\begin{tikzcd}
\Tmr \PP(\mr{V})|_\ell \rar[ dashed] \dar[ dashed ] & \Res^\Gmr_\Hmr \Tmr\PP(\Vmr)  \dar \\
\ell \rar &  \Res^\Gmr_\Hmr \PP(\Vmr),
\end{tikzcd}
\]
is an $\Hmr$-representation. By  \Cref{defn:homogeneous} (also see \Cref{rmk:GIhomogeneity}),  it is enough to show that  
\[ \Tmr \PP(\mr{V})|_\ell \cong n \Res^\Gmr_\Hmr \uprho_{\Gmr} -1 = n |\Gmr/\Hmr| \uprho_{\Hmr} -1 \]
regardless of the path component of $\ell$.

The path components of $\mr{H}$-fixed points of $\PP(\Vmr)$ are in one-to-one correspondence with the isomorphism classes of 1-dimensional sub-representations of $\Res^{\Gmr}_\Hmr \Vmr$, i.e., 
\[ 
\pi_0 (\PP(\Vmr)^\Hmr) = \pi_0^{\Hmr} \PP(\Vmr)  \cong \{ \Lmr \subset \Res^{\Gmr}_\Hmr \Vmr: \dim_\RR \Lmr =1  \}/\mr{iso}. 
\]
Further,  the classical argument of Atiyah \cite[Lemma 4.5]{AtiyahTC} generalizes equivariantly to show that  \[ \Res^\Gmr_\Hmr \Tmr \PP(\Vmr) \oplus \upepsilon_1 \cong \Hom (\upgamma_1, \upepsilon_{\mr{V}}),\]
 where $\upgamma_1$ is the tautological line bundle over $\PP(\Vmr)$ whose total space is 
 \[ \mr{Tot}(\upgamma_1):=  \{(\ell, {\sf v}): {\sf v} \in \ell  \} \subset \PP(\mr{V})  \times \mr{V} \]
 as usual, and $\upepsilon_1$ and $\upepsilon_{\Vmr}$ are trivial bundles of dimension $1$ and $\mr{V}$ respectively. 
 
  Therefore, if $\ell \in \PP(\Vmr)^\Hmr$ belongs to the component corresponding to the $1$-dimensional $\mr{H}$-representation $\Lmr$ and  $\mr{V} = n \uprho_{\Gmr}$, then 
 \[ \Tmr \PP(\mr{V})|_\ell +1 = \Lmr^{\ast} \otimes \Res^\Gmr_{\Hmr}\mr{V} = \Lmr^{\ast} \otimes n
   |\Gmr/\Hmr| \uprho_\Hmr = n |\Gmr/\Hmr| \uprho_\Hmr   \in \RO(\Gmr)   \]
 is independent of the choice of $\Lmr$,  where $\Lmr^{*}$ is the dual of $\Lmr$.
\end{proof}
\begin{rmk} \label{rmk:TP(V)} 
There are plethora of  examples where $\mr{V}$ is not a sum of regular $\Gmr$-representations and has  more than one isomorphic class of $1$-dimensional sub-representations yet $\Tmr \PP(\Vmr)$  is homogeneous. 
For instance,  when $\Gmr = \mr{C}_4$ and $\upsigma$ is  the sign representation of $\mr{C}_4$, an analysis identical to that  in the proof of \Cref{lem:P(V)}  will show that 
$\Tmr \PP(n(1 \oplus \upsigma) )$  is a homogeneous bundle of dimension
$(n-1)+n \upsigma$ for all positive integers $n$. 
\end{rmk}

\subsection{Relative homogeneity of vector bundles} \label{subsec:relhomogeneity} \ 

The notion of homogeneity can be defined relative to any $\EE_\infty^{\mr{G}}$-ring spectrum $\mr{R}$ using genuine stable $\mr{J}$-homomorphism $\mJJ_{\mr{G}}$ (as in \Cref{rmk:genuineJ}) and $\pic(\R)$, the equivariant Picard spectrum associated to $\R$. The spectrum $\pic(\R)$ is a genuine $\mr{G}$-spectrum defined by that property that its $\mr{H}$-fixed points are the delooping of the $\mr{H}$-$\infty$-groupoid  spanned by invertible $\upiota_{\mr{H}}(\R)$-modules, where $\upiota_{\mr{H}}: \Sp^{\mr{G}} \longrightarrow \Sp^{\mr{H}}$ is the restriction functor. Thus,  
\[ \pic(\R)^{\mr{H}} \simeq \pic(\upiota_{\mr{H}}(\R) )\]
for any  $\EE_\infty^{\mr{G}}$-ring spectrum $\mr{R}$. Unfortunately, a construction of such a $\mr{G}$-equivariant spectrum has not appeared in the literature thus far, hence we sketch one in the next subsection.

For any $\mr{G}$-equivariant vector bundle $\upxi$ over $\mr{B}$, let ${\sf f}_{\upxi, \mr{R}}$ denote  the composite 
\begin{equation} \label{vectrelR}
\begin{tikzcd}
{\sf f}_{\upxi, \mr{R}}: \mr{B} \rar["{\sf f}_{\upxi}"] & \Omega^{\infty}_{\mr{G}} \ko_{\mr{G}} \ar[rr, "\Omega^{\infty}_{\mr{G}}\mJJ_{\mr{G}}"] && \Omega^{\infty}_{\mr{G}} {\pic(\SS_{\mr{G}})} \ar[rr,"\Omega^{\infty}_{\mr{G}}\pic(\iota_{\mr{R}})"] &&\Omega^{\infty}_{\mr{G}} {\pic(\mr{R})},
\end{tikzcd}
\end{equation}
where $\pic(\SS_{\mr{G}})$ and $\pic(\R)$ are the Picard spectrum of
$\SS_{\mr{G}}$ and $\R$ respectively (see \Cref{subsec:equivPic}), and $\iota_{\R}$ is the unit map of $\R$. 
\begin{defn}  \label{defn:relhomogeneous} A $\mr{G}$-equivariant vector bundle  $\upxi$ is  \emph{homogeneous relative to  $\R$}  (abbrev.  \emph{$\R$-homogeneous}) if $\underline{\pi}_0( {\sf f}_{\upxi, \mr{R}})$ factors 
\begin{equation} \label{eqn:factor}
 \begin{tikzcd}
&&& {\sf c}_{\All_{\mr{G}}} \dar["\mr{p}"] \\
\underline{\pi}_0(\mr{B}) \ar[rrru, dashed] \ar[rrr,"\underline{\pi}_0({\sf f}_{\upxi, \mr{R}})"'] &&&\underline{\pi}_0(  \pic(\R))
\end{tikzcd}
\end{equation}
through an $\All_{\mr{G}}$-point $\mr{p}$ of $\underline{\pi}_0 (\pic(\R))$.
 We call $\mr{p}$ a \emph{coordinate of $\R$-homogeneity}, and set 
   $\mr{p}_{\mr{G}}$ as the \emph{relative dimension}  
 \[\dim_\R (\upxi, \mr{p}) :=  \mr{p}_{\mr{G}}
\]
of the pair $(\upxi, \mr{p})$ with respect to $\R$.
 \end{defn}
\begin{rmk} \label{rmk:homVSrelhom} Any homogeneous bundle is automically $\R$-homogeneous for any $\EE_\infty^{\mr{G}}$-ring spectrum $\R$. Conversely,  an $\R$-homogeneous bundle $\upxi$ is homogeneous if and only if  the relative coordinate admits a lift along $\ull{\pi}_0(  \pic(\iota_\R) \circ \mJJ_{\mr{G}} )$
\begin{equation} \label{liftdiag}
\begin{tikzcd}
 & \ull{\pi}_0 (\ko_{\mr{G}}) \ar[d, "\ull{\pi}_0( \pic(\iota_\R) \circ \mJJ_{\mr{G}} )"] \\
 {\sf c}_{\All_{\mr{G}}} \ar[ur, dashed]  \ar[r, "\mr{p}"'] &  \ull{\pi}_0 \pic(\mr{R}). 
\end{tikzcd}
\end{equation}
 However, the map $\ull{\pi}_0\mJJ_{\mr{G}}$ of Mackey functors is not well understood for a general group  $\mr{G}$ (but see \cite{AngelPic, PicFLM} for partial answers). 
\end{rmk}
\begin{rmk} \label{rmk:liftrelhom} Note that the diagram in
    \eqref{liftdiag} automatically admits a lift if the vertical map is injective  for all
    $\mr{G}/\mr{H} \in \cal{O}_\Gmr$. This is the case when $\Gmr = \mr{C}_n$ for
    $1\leq n\leq 4$ or $n=6$, and  $\Rmr= \SS_\Gmr$. 
  This map is not injective when $\mr{G}=\mr{C}_n$ and $\mr{R}=
    \SS_\Gmr$ for $n=5$ or $n \geq 7$ (follows from \cite[Theorem A]{AngelPic}), or when $\Gmr$ is a finite group with
    a cyclic quotient group and $\mr{R}$ is a $\Gmr$-equivariant complex oriented ring spectrum  (see
      \Cref{rmk:complexorient}). In these cases, $\mr{R}$-homogeneity is more
      general than homogeneity.
\end{rmk}

\subsection{Equivariant Picard spectrum and the stable J-homomorphism} \label{subsec:equivPic} \

The Picard group of a symmetric monoidal category $\mathcal{C}$ is defined as the isomophism classes of invertible objects.  When $\mathcal{C}$ is the symmetric monoidal category of $\mr{R}$-modules for an $\EE_\infty$-ring $\mr{R}$, 
Mike Hopkins refined the definition of Picard group  to obtain a Picard spectrum $\pic(\mr{R})$  such that the $\pi_0(\pic(\mr{R}))$ is precisely the Picard group of the homotopy category of $\mr{R}$-modules. Therefore, it is natural to ask for a $\mr{G}$-equivariant Picard spectrum  $\pic(\mr{R})$ associated to  an $\EE_\infty^{\mr{G}}$-ring spectrum  $\mr{R}$.  

The recent paper \cite{HHKWZ} uses the language of parametrized $\infty$-category to refine the definition of the equivariant Picard group (see \cite{PicFLM}) to construct the equivariant Picard spectrum $\pic(\SS_{\mr{G}})$.  
We will therefore briefly review the  foundations of parametrized $\infty$-category necessary  before generalizing  \cite{HHKWZ} to construct a $\mr{G}$-spectrum $\pic(\mr{R})$ for all $\EE_\infty^{\mr{G}}$-ring spectrum $\mr{R}$.

Recall that a $\mr{G}$-$\infty$-category is a coCartesian fibration  \cite{Expose1, Expose2} 
\[ 
\begin{tikzcd}
{\sf p}: \underline{\mathcal{C}} \rar & \mathcal{O}_{\mr{G}}^{\mr{op}},
\end{tikzcd}
\] 
and models genuine $\mr{G}$-objects. A $\mr{G}$-symmetric monoidal $\mr{G}$-$\infty$-category is a coCartesian fibration over $\FG$ (see \cite[$\mathsection$2]{HHKWZ} for a definition)
\[ 
\begin{tikzcd}
{\sf p}: \underline{\mathcal{M}}^{\otimes} \rar & \FG
\end{tikzcd} 
\]
satisfying the equivariant Segal condition (in \cite[Definition 2.1.8]{HHKWZ}).   The maximal $\infty$-groupoid spanned by invertible objects within the fiber  $\ull{\cal{M}}_{[\mr{G}/\mr{H}]}$ (over
 $[\mr{G}/\mr{H}] \in \FG$)  assemble to form a
 $\mr{G}$-$\infty$-groupoid $ \ull{\mr{Pic}}(\ull{\cal{M}})$ called the Picard
 groupoid of $\ull{\cal{M}}$. It inherits a $\mr{G}$-symmetric monoidal
 structure from $\ull{\cal{M}}^{\otimes}$, and therefore,
 $\ull{\mr{Pic}}(\ull{\cal{M}})$  models an $\EE_\infty^{\mr{G}}$-space (see
 \cite[$\mathsection$5.1, Example A.4.3]{HHKWZ}) which
 deloops to a genuine $\mr{G}$-spectrum using the functor $\mr{B}$ in \cite[A.4.1]{Expose4} (which is a version  of an equivariant delooping machine originally introduced in \cite{GMequivK}).

 Thus, to complete the construction of $\pic(\mr{R})$,  it remains to show:
   \begin{prop}
     There exists a $\mr{G}$-symmetric monoidal $\mr{G}$-$\infty$-category of
     $\Rmr$-modules.
   \end{prop}
\begin{proof}
Firstly, there is a functor from $\cal{O}_\mr{G}^{\mr{op}}$  to the category of
simplicial symmetric
monoidal categories that sends $\mr{G}/\mr{H}$ to the category of
$\upiota_{\mr{H}}\mr{R}$-modules.
 Moreover, for subgroups $\mr{K} \subset \mr{H} \subset \mr{G}$, there is a relative norm functor $\Nmr_\Kmr^{\Hmr,\upiota_\mr{K} \mr{R}}$ from $\upiota_{\mr{K}} \mr{R}$-modules to $\upiota_{\mr{H}} \mr{R}$-modules
 sending $\mr{M}$ to $\upiota_{\mr{H}} \mr{R} \sma_{\Nmr_\Kmr^\Hmr \upiota_{\mr{K}} \mr{R}} \Nmr_\Kmr^\Hmr (\mr{M})$.
These can be assembled into a functor from $\FG$ to simplicial categories and
then the Grothendieck construction will turn it into a simplicial category
$\underline{\mr{Mod}}_\Rmr^{\otimes}$ over $\FG$. By contruction, $\underline{\mr{Mod}}_\Rmr^{\otimes}$ is a fully split simplicial
Grothendieck op-fibration (see \cite[Sec 2.1]{Bonventre}), and of Segal type.
The operadic nerve of $\underline{\mr{Mod}}_\Rmr$ is then a $\Gmr$-symmetric
monoidal $\Gmr$-$\infty$-category by \cite[Prop 5.34]{Bonventre}.  
\end{proof}

 \begin{rmk}[Genuine stable $\mr{J}$-homomorphism]  \label{rmk:genuineJ} 
Let $\cal{R}\mr{ep}$ denote the  symmetric monoidal topological category  whose objects are $\mathbb{R}^n$ and morphisms are $\bigsqcup_{n \in \mathbb{N}} \mr{Gl}_n(\mathbb{R})$. One can consider ``topological $\mr{G}$-objects'' in $\cal{R}\mr{ep}$  \cite[Construction 2.2.3]{HHKWZ} and obtain a $\mr{G}$-symmetric monoidal $\mr{G}$-$\infty$-category
\[ 
\begin{tikzcd}
\ull{\cal{R}\mr{ep}}^{\otimes} \rar &  \ull{\mathrm{Fin}}_{*}^{\mr{G}}
\end{tikzcd}
\]
which deloops to  $\ko_{\mr{G}}$. Explicitly, the fiber $\infty$-category
$\ull{\cal{R}\mr{ep}}_{[\mr{G/H}]}$ is equivalent to the nerve of
the category of finite dimensional $\mr{H}$-representations and $\mr{H}$-equivariant
linear isomorphisms. 
By sending an $\mr{H}$-representation $\mr{V}$ to $\Sigma^{\infty}_{\mr{H}}\mr{S}^{\mr{V}}$ via one-point compactification, we obtain the $\mr{G}$-equivariant $\mr{J}$-homomorphism $\mJJ_{\mr{G}}$ (also see \cite[$\mathsection$1]{McClureJG}).
\end{rmk}

\subsection{Orientation  and Thom isomorphism} \label{subsec:orient-thom} \ 

The classifying map of a  $\mr{G}$-equivariant vector bundle $\upxi$ can be regarded as a map of $\mr{G}$-$\infty$-groupoids
\[ 
\begin{tikzcd}
\mr{F}_{\upxi}:\mr{B} \rar & \ull{\cal{R}\mr{ep}}_{\mr{G}},
\end{tikzcd}
\] 
where $\mr{B}$ is the  fundamental $\mr{G}$-$\infty$-groupoid of the base space of $\upxi$. In fact,  ${\sf f}_{\upxi}$  (as in  \Cref{notn:classify}) is  the homotopy class of $\mr{F}_{\upxi}$. Likewise, the map ${\sf f}_{\upxi, \R}$ defined in \eqref{vectrelR} can also be constructed as a map  of $\mr{G}$-$\infty$-groupoids 
\[ 
\begin{tikzcd}
\mr{F}_{\upxi, \mr{R}}: \mr{B} \rar & \ull{\cal{R}\mr{ep}}_{\mr{G}}  \rar &
\ull{\mr{Pic}}(\ull\Sp^{\mr{G}}) \rar & \ull{\mr{Pic}}(\ull{\mathcal{M}\mr{od}}_\mathrm{R}).
\end{tikzcd}
\]

\begin{defn} \label{defn:orientation}
 An \emph{$\R$-orientation} of  a $\mr{G}$-equivariant vector bundle $\upxi$  is a homotopy between the composite $\mr{F}_{\upxi, \mr{R}}$ and 
 a constant map. 
\end{defn} 
\begin{rmk}[A comparison with the nonequivariant orientation theory]
Note that when $\Gmr$ is the trivial group then the  dimension of the
  vector bundle $\upxi$ determines the co-ordinate of homogeneity. Thus by
  choosing the base point of  $\Omega^\infty \pic(\Rmr)$ in the component corresponding
  to $\dim \upxi$, the map ${\sf f}_{\upxi, \Rmr}$ (which is the completion of
  $\mr{F}_{\upxi, \Rmr}$)  lifts to $\Omega^{\infty}\bgl(\Rmr)$,  the $0$-connected cover  of $\Omega^\infty \pic(\Rmr)$ (according to the notations of  \Cref{rem:bgl}). From hereon it is straightforward to see that our definition of equivariant orientation, i.e. \Cref{defn:orientation}, when restricted to the trivial group coincides with the nonequivariant orientation theory of  May-Quinn-Ray-Tornehave \cite{MQRT} (also see \cite{ABGHR}).
\end{rmk}

The  $\mr{R}$-Thom spectrum  of $\upxi$ is defined as the colimit
 \begin{equation} \label{eqn:RThom}
 \mr{Th}(\upxi, \mr{R}) := \Gcolim_{ \ \ \mr{B}} \ \mr{F}_{\upxi, \mr{R}} 
 \end{equation}
 as in \cite[$\mathsection$9]{Expose2}. We  abbreviate  $\mr{Th}(\upxi) := \mr{Th} (\upxi, \SS_{\mr{G}})$.
The commuting diagram of $\mr{G}$-$\infty$-categories
\[
\begin{tikzcd}
\ull{\mr{Pic}}(\ull\Sp^{\mr{G}}) \rar[] \dar[hook] & \ull{\mr{Pic}}( \ull{\mathcal{M}\mr{od}}_\mr{R}) \dar[hook] \\
\ull{\Sp}^{\mr{G}} \rar[, "-\sma \R"'] &  \ull{\mathcal{M}\mr{od}}_\mr{R}
\end{tikzcd}
 \]
implies there is a map from $\mr{Th}(\xi, \mr{R})$  to $\mr{Th}(\upxi) \sma \mr{R}$. In fact, this map is a weak equivalence as the smash product commutes with colimits.


\bigskip
\begin{proof}[\textbf{Proof of} \Cref{main1} ] Let $\cal{I}$ denote the dimension of $\upxi$ relative to $\R$.  The proof of  $\textit{(1)} \Leftrightarrow \textit{(2)}$ is essentially a tautology.  

The proof of $\textit{(2)} \Leftrightarrow \textit{(3)}$ is classical: 
An $\R$-Thom class ${\sf u}_{\upxi} $ leads to an $\R$-Thom isomorphism
\[ 
\begin{tikzcd}
\omega: \mr{Th}(\upxi) \sma \R \ar[rr, "\Delta \sma 1_{\R}"] && \mr{Th}(\upxi) \sma \mr{B}_+ \sma \R \rar["(2)"] & \cal{I}  \sma \mr{B}_+ \sma \R \rar["(3)"] & \mr{B}_+ \sma \cal{I} ,  
\end{tikzcd}
\]
where $\Delta$ is the Thom diagonal map, the map $(2)$ is induced by  ${\sf u}_{\upxi}$, and the map $(3)$ is induced by the $\mr{R}$-module structure of $\cal{I}$. Conversely, an $\R$-Thom isomorphism  $\omega$ leads to an $\R$-Thom-class 
\[ 
\begin{tikzcd}
{\sf u}_{\upxi}:\mr{Th}(\upxi)  \ar[r, "(4)"] & \mr{Th}(\upxi) \sma   \R \rar["\omega"] &  \mr{B}_+ \sma \cal{I}  \rar["(6)"] & \cal{I}, 
\end{tikzcd}
\]
where the map $(4)$ is induced by the unit of $\R$, and the map $(6)$ collapses $\mr{B}$ to a point. 
\end{proof}
\begin{rmk} \label{rmk:OimplyH} The existence of an $\R$-orientation of a $\mr{G}$-equivariant bundle $\upxi$
  automatically implies $\R$-homogeneity. If ${\sf c}$ denote a constant map equivalent to $\mr{F}_{\upxi, \R}$, 
then 
\[ 
\begin{tikzcd}
\ull{\pi}_0({\sf c}): \ull{\pi}_0(\mr{B})  \rar & \underline{\pi}_0( \Omega^{\infty}_{\mr{G}} \pic(\R))
\end{tikzcd}
\] 
serves as a coordinate of $\R$-homogeneity of $\upxi$.
\end{rmk}
\begin{ex} For any finite group $\Gmr$,
     the $\mr{G}$-equivariant complex tautological line bundle \[ 
\upgamma_1^\CC := \begin{tikzcd}
\mr{E}_{\Gmr}\mr{S}^1 \times_{\mr{S}^1} \upphi_1 \dar \\
\mr{B}_{\Gmr}\Smr^1
\end{tikzcd}
\]
admits an $\mr{R}$-Thom class for any complex orientable $\EE_\infty^{\Gmr}$-ring $\Rmr$  (see \cite[Definition 9.5]{GreenleesFGL}). Here
  $\upphi_1$ is the $1$-dimensional complex representation of $\Smr^1$ induced
  by its  inclusion in $\CC^\times$. Note that $\upgamma_1^\CC$  is
    not homogeneous when $\Gmr$ has more than one isomorphic classes of
  $1$-dimensional complex representation, but it is $\mr{R}$-homogeneous. In fact, the $\mr{R}$-homogeneity is forced by
  the existence of a $\Rmr$-Thom class as explain in \Cref{rmk:OimplyH}. Examples  of $\Gmr$-equivariant complex oriented ring spectrum include the $\Gmr$-equivariant periodic complex $\mr{K}$-theory  $\KU_{\Gmr}$, the universal complex oriented theory $\mr{MU}_\Gmr$, among others (see \cite{GreenleesFGL}). 
\end{ex}
\begin{rmk}(Equivariant complex orientations and
    homogeneity)  \label{rmk:complexorient}  According to 
    \cite[Definition 1.2]{Okonek}, an equivariant cohomology theory ``has Thom classes'' 
    if   each $\Gmr$-equivariant complex vector bundle can be assigned an
    $\mr{R}$-Thom class which is compatible across cartesian products and pullbacks. Equivariant cohomology theories with Thom classes are closely related to equivariant complex oriented cohomology theories as explained in \cite[$\mathsection$6]{CGK}.  Suppose $\Rmr$ is an $\EE_\infty^{\Gmr}$-ring spectrum representing a cohomology theory which has Thom classes. Then by restricting 
    the $\Rmr$-Thom isomorphism of the universal ${\sf n}$-plane complex vector bundle
    over $\mr{B}_\Gmr \Umr({\sf n})$  to various fixed points, we get
     that for any ${\sf n}$-dimensional complex $\Gmr$-representation $\sf r$,
      there is an equivalence
\begin{equation*}
\Sigma^{\sf r} \Rmr \simeq \Sigma^{2{\sf n}} \Rmr,
\end{equation*}
and consequently, the map 
\[
\begin{tikzcd}
 \pi_0^{\Gmr}(\ku_\Gmr) \rar[hook] &  \pi_0^{\Gmr}(\ko_\Gmr) \ar[rr, "\pi_0^\Gmr\mJJ_{\mr{G}}"] && \pi_0^{\Gmr} \pic(\SS_\Gmr) \rar& \pi_0^{\Gmr} \pic(\Rmr)
\end{tikzcd}
\]
contains ${\sf r} -  \dim_{\mathbb{C}} {\sf r}   \in \mr{RU}(\Gmr)$ in the kernel.
\end{rmk}

In the early nineties, Costenoble-Waner \cite{CW} obtained an equivariant Thom isomorphism  theorem for equivariant bundles after twisting the ordinary cohomology (with Burnside coefficients) using the fundamental groupoid $\Pi \mr{B}$ of the base space $\mr{B}$ of the given bundle. Consequently, the underlying cohomology theory is graded using the representation ring $\mr{RO}(\Pi \mr{B})$. In some sense, the work of Costenoble-Waner is an extension of the classical Thom isomorphism theorem  \cite[Lemma 2.1]{AtiyahTC} which involves local coefficients\footnote{We thank  Agnes Beaudry for explaining this perspective.}. The main advantage of this theory is there are no requirements on the bundle (such as homogeneity)  to obtain a Thom class or a Thom isomorphism. This is also a disadvantage as the theory cannot discern oriented bundles from those that are not oriented geometrically (as per \cite[Definition 7.9]{CMW}). This led May \cite{MayOrient} to propose an alternative theory where an orientation of an equivariant bundle  is represented by a family of  coherent Thom classes, however having multiple Thom classes is not ideal for  computational purposes.

In our theory, we take advantage of the notion of homogeneity, so that our definition of orientation agrees with the geometric interpretation of equivariant orientation \cite{MayOrient} (after setting $\mr{R} = \mr{H} \ull{\Acal}_{\Gmr}$) at the same time represented by a single Thom class even when the base space is disconnected. This leads us to new applications which are discussed in the next section.

\section{The first equivariant Stiefel--Whitney class} \label{sec:SW}
For an $\R$-homogeneous bundle $\upxi$ of relative dimension $\cal{I}$, the corresponding relative dimension zero virtual bundle  is classified by the map
\[ 
\begin{tikzcd}
{\sf f}_{\upxi, \R} - {\sf c}_{\mathcal{I}}:  \mr{B} \rar & \Omega^{\infty}_{\mr{G}}{\pic(\R)},
\end{tikzcd}
\] 
where  ${\sf c}_{\cal{I}}$ is the constant map at $\cal{I}$.  By construction, the adjoint of this map admits a lift
\begin{equation} \label{dimzero}
\begin{tikzcd}
&& \bgl(\R) \dar \\
\Sigma^{\infty}_{\mr{G}} \mr{B}_+ \ar[rr, "{\sf f}_{\upxi, \R} - {\sf c}_{\mathcal{I}}"']   \ar[urr, dashed, "{\sf f}^{(0)}_{\upxi, \R}"] && \text{$\pic(\R)$},
\end{tikzcd}
\end{equation}
where $\bgl(\R)$ is the $0$-connected cover of $\pic(\R)$ (see \Cref{rem:bgl}).  

Notice that when $\Gmr$ is the trivial group and $\mr{R} = \Hmr\ZZ$ then $\bgl(\Rmr) \simeq \Sigma \Hmr \FF_2$, and the homotopy class
\[ [{\sf f}_{\upxi, \Hmr\ZZ}] \in \mr{H^1}(\mr{B}; \FF_2) \] 
is precisely the first Stiefel-Whitney class of $\upxi$. Therefore, we make the following definition.
\begin{defn} \label{defn:SW} For an $\R$-homogeneous $\mr{G}$-equivariant bundle $\upxi$  define its  \emph{first  Steifel Whitney class with respect to $\R$}  as the homotopy class
\[ {\sf w}_1^{\R}(\upxi):= [{\sf f}^{(0)}_{\upxi, \R}] \in
  [\Sigma^{\infty}_{\mr{G}} \mr{B}_+, \bgl(\R)]^{\mr{G}}. \] 
\end{defn}
\bigskip
\begin{rmk} \label{rem:bgl} Let $\ull{\mr{BGL}}_1( \mr{R}\text{-}\mr{Mod})$ denote
  the sub $\mr{G}$-$\infty$-groupoid of  $\ull{\mr{Pic}}( \mr{R}\text{-}\mr{Mod})$
  spanned by $\R$.
  Then, $\ull{\mr{BGL}}_1( \mr{R}\text{-}\mr{Mod})$ is $\mr{G}$-symmetric monoidal and deloops to a genuine $\mr{G}$-equivariant spectrum, which we call $\bgl(\R)$. It is easy to see that $\bgl(\R)$ is the fiber of the zeroth Postnikov approximation map 
\[ 
\begin{tikzcd}
{\pic(\R)} \rar & \mr{H} \ull{\pi}_0 \pic(\R),
\end{tikzcd}
\]
and  hence, it is the $0$-connected cover of $\pic(\R)$. 
\end{rmk}
\begin{rmk} \label{rmk:GL1}
Emulating the classical work of \cite{MQRT} one may define the $\mathcal{O}_{\mr{G}}$-space $\ull{\mr{GL}}_1(\mr{R})$ as the pullback 
\[ 
\begin{tikzcd}
\ull{\mr{GL}}_1(\R) \dar[ dashed] \rar[ dashed] & \Omega^{\infty}_{\mr{G}} \R \dar[] \\
\ull{\pi}_0(\R)^{\times} \rar[] & \ull{\pi}_0(\R).
\end{tikzcd}
\]
For an $\EE_\infty^{\mr{G}}$-ring spectrum $\R$,  $\ull{\mr{GL}}_1(\R)$ is
  equivalent to the fixed points of an $\EE_\infty^{\mr{G}}$-space  and   deloops to
  a genuine $\mr{G}$-spectrum $\gl(\R)$ (details to appear in
  \cite{KMZpresheaf})\footnote{A construction of the unit space in the context of equivariant Segal spaces appears in \cite{Rekha}.}.
It can be shown that \[ \bgl(\R) \simeq \Sigma \gl(\R), \]
and  therefore,  
\[ 
\ull{\pi}_{\sf t} (\bgl(\R)) \cong \left \lbrace 
\begin{array}{lll}
\ull{\pi}_0(\R)^{\times} & \text{if  ${\sf t} =1$,  } \\
\ull{\pi}_{{\sf t} -1}(\R) & \text{if ${\sf t} \geq 2 $, and} \\
0 & \text{otherwise,}
\end{array}
\right. 
\]
where ${\sf t} \in \ZZ$.

\end{rmk}
\begin{rmk}[Uniqueness of ${\sf w}_1^{\R}(-)$] \label{rmk:uniquew1}
The homotopy class of the map ${\sf f}^{(0)}_{\upxi, \R}$ is well-defined as the indeterminacies lie in $\mr{H}^{-1}(\mr{B}; \ull{\pi}_0\pic(\R))$  (see \Cref{rem:bgl}), which is a trivial group for any $\mr{G}$-space $\mr{B}$. 
\end{rmk}

\bigskip
\begin{proof}[\textbf{Proof of} \Cref{thm:w1orient}] \label{proof:w1orient}
By definition,  ${\sf w}_1^{\R}(\upxi)$ is the obstruction to $\R$-orientability of an $\R$-homogeneous bundle $\upxi$. 
\end{proof}   
\subsection{  $\mr{H}\ull{\ZZ}$-orientability of homogenous bundles } \ 

It follows from \Cref{rmk:GL1} that $\bgl(\R) \simeq \Sigma\mr{H} \cal{T}^{\times}$ when $\R = \mr{H}\mathcal{T}$ for any Tambara functor $\mathcal{T}$. In particular, 
when $\R = \mr{H}\ull{\ZZ}$ then $\bgl(\R) \simeq \Sigma\mr{H} \ull{\FF}_2$. Therefore, the Stiefel--Whitney class ${\sf w}_1^{\ull{\ZZ}}(\upxi)$ for an $\mr{H}\ull{\ZZ}$-homogeneous bundle $\upxi$ is an element in
\[
{\sf w}_1^{\ull{\ZZ}}(\upxi) \in \mr{H}^1(\mr{B}; \ull{\FF}_2),
 \]
where $\mr{B}$ is the base space of $\upxi$.  Thus, when $\mr{H}^1(\mr{B}; \ull{\FF}_2) = 0$, $\upxi$ is automatically $\mr{H}\ull{\ZZ}$-oriented.  

\bigskip
\begin{proof}[\textbf{Proof of} \Cref{thm:HZtwofold}]  For an $\mr{H}\ull{\ZZ}$-homogeneous vector bundle $\upxi$,  the additivity formula  \eqref{eqn:additivity} implies
\[ {\sf w}_1^{\ull{\ZZ}}( \upxi \oplus \upxi) = {\sf w}_1^{\ull{\ZZ}}( \upxi) + {\sf w}_1^{\ull{\ZZ}}( \upxi) = 2 {\sf w}_1^{\ull{\ZZ}}( \upxi) = 0    \]
as $\mr{H}^1(\mr{B}; \ull{\FF}_2)$ is an $\FF_2$-vector space. Thus, by \Cref{thm:w1orient}, $2$-fold direct sum  of any $\mr{H}\ull{\ZZ}$-homogeneous vector bundle  is automatically $\mr{H}\ull{\ZZ}$-orientable. 
\end{proof}

\subsection{  $\mr{H}\ull{\cal{A}}_{\mr{G}}$-orientability of homogenous bundles
}

For a finite group $\mr{G}$,  all nontrivial  elements in
  ${\cal{A}}_{\mr{G}}^{\times}$  have order $2$ \cite{Matsuda} (also see \Cref{thm:Matsuda}). Therefore, $\mr{H}^\star_{\mr{G}}(\mr{B}; \ull{\cal{A}}_{\mr{G}}^{\times})$ is a graded $\mathbb{F}_2$-vector space. 

\bigskip
\begin{proof}[\textbf{Proof of} \Cref{thm:HAtwofold}] Let $\upxi$ be an $\mr{H}\cal{A}_{\Gmr}$-homogeneous bundle. Then the additivity formula \eqref{eqn:additivity}, implies  
\[ {\sf w}_1^{\ull{\cal{A}}_{\mr{G}}}( \upxi \oplus \upxi) = {\sf w}_1^{\ull{\cal{A}}_{\mr{G}}}( \upxi) + {\sf w}_1^{\ull{\cal{A}}_{\mr{G}}}( \upxi) = 2 {\sf w}_1^{\ull{\cal{A}}_{\mr{G}}}( \upxi) = 0.   \]
  Thus, by \Cref{thm:w1orient}, $2$-fold direct sum of any  $\upxi$ is $\mr{H}\cal{A}_{\Gmr}$-orientable. 
 \end{proof}

\bigskip 
\begin{rmk} \label{rmk:2foldT}A Tambara functor $\mathcal{T}$ is an
  $\ull{\cal{A}}_{\mr{G}}$-algebra. Therefore, one may wonder  if \Cref{thm:HAtwofold} implies $2$-fold direct sum of an $\mr{H}\mathcal{T}$-homogeneous
  bundle is always $\mr{H}\mathcal{T}$-orientable. While the statement is true when $\mr{G}$ is trivial, it does not generalize to nontrivial groups  because an $\mr{H}\cal{T}$-homogeneous bundle need not be $\mr{H}\ull{\cal{A}}_{\mr{G}}$-homogeneous. However, the authors are not aware of any counterexample. 
\end{rmk}
\bigskip
\begin{proof}[\textbf{Proof of} \Cref{thm:AorientGodd}]
The natural map $ 
\upiota: \ull{\cal{A}}_{\mr{G}} \longrightarrow \ull{\ZZ} 
$ of Tambara functors 
induces a map of Mackey functors   on the units
\begin{equation} \label{AGtoZunit}
\begin{tikzcd}
\upiota^{\times}: \ull{\cal{A}}_{\mr{G}}^{\times} \rar & \ull{\ZZ}^{\times} \cong \ull{\FF}_2
\end{tikzcd}
\end{equation}
which is an isomorphism when $|\mr{G}|$ is odd (see \Cref{thm:Matsuda}). Consequently,  an $\ull{\cal{A}}_{\mr{G}}$-homogeneous bundle is $\ull{\cal{A}}_{\mr{G}}$-orientable if and only if  it is $\mr{H}\ull{\ZZ}$-orientable.  
\end{proof}

\bigskip
When  $|\mr{G}|$ is even,  \eqref{AGtoZunit} is not an isomorphism, and  therefore, 
\[ \ull{\cal{A}}_{\mr{G}}^{\circ} := \ker(\upiota^{\times})\]
is a nontrivial Mackey functor (see \Cref{Appendix:unit}). It follows that: 
\begin{prop}
 The fiber  of  the map 
\[ 
\begin{tikzcd}
\bgl(\upiota): \Sigma \mr{H } \ull{\cal{A}}_{\mr{G}}^{\times} \simeq \bgl(\mr{H}\ull{\cal{A}}_{\mr{G}}) \ar[rr] && \bgl( \mr{H}\ull{\ZZ}) \simeq \Sigma\mr{H} \ull{\FF}_2
\end{tikzcd}
\]
is equivalent to $\Sigma\mr{H}\ull{\cal{A}}_{\mr{G}}^{\circ}$.
\end{prop}
\begin{proof} Since $\upiota^{\times}$ in \eqref{AGtoZunit}  is a surjection,  the long exact sequence
\[ 
\begin{tikzcd}
  \dots   \rar[] &  \ull{\pi}_{{\sf t}+1} \bgl( \mr{H}\ull{\ZZ})   \rar[]  & \ull{\pi}_{{\sf t} +1}(\bgl( \mr{H}\ull{\ZZ})  )  \ar[dll, out= east, in=west, looseness=1.8, overlay]\\
 \ull{\pi}_{\sf t} (\mr{Fib}(\bgl(\upiota)))\rar[] & \ull{\pi}_{\sf t} \bgl( \mr{H}\ull{\ZZ}) \rar[]  &\dots 
\end{tikzcd}
\]  
of Mackey functors implies
\[ 
\ull{\pi}_{\sf t} (\mr{Fib}(\bgl(\upiota))) \cong \left \lbrace 
\begin{array}{lll}
\ull{\cal{A}}_{\mr{G}}^{\circ} & \text{if  ${\sf t} =1$, and  } \\
0 & \text{otherwise,}
\end{array}
\right. 
\]
where ${\sf t} \in \ZZ$. Thus,  $\mr{Fib}(\bgl(\upiota)) \simeq \Sigma\mr{H} \ull{\cal{A}}_{\mr{G}}^{\circ}$. 
\end{proof}
When an $\mr{H}\ull{\cal{A}}_{\mr{G}}$-homogeneous  $\mr{G}$-equivariant
vector bundle $\upxi$  is $\mr{H}\ull{\mathbb{Z}}$-orientable,  there exists a lift in the diagram
\begin{equation} \label{fiberg}
\begin{tikzcd}
& \Sigma^{\infty}_{\mr{G}}\mr{B}_+ \ar[dl, dashed, bend right ,"{\sf f}^{{\sf g}}_{\upxi, \mr{H} \ull{\cal{A}}_{\mr{G}} } "' ] 
\ar[d, "{\sf f}^{(0)}_{\upxi, \mr{H} \ull{\cal{A}}_{\mr{G}} } "']
\ar[drr, "{\sf f}^{(0)}_{\upxi, \mr{H} \ull{\ZZ} }"]  \\ 
 \Sigma\mr{H} \ull{\cal{A}}_{\mr{G}}^{\circ} \rar & \bgl(\mr{H} \ull{\cal{A}}_{\mr{G}}) \ar[rr]  && \bgl(\mr{H} \ull{\ZZ}) 
\end{tikzcd}
\end{equation}
because ${\sf f}^{(0)}_{\upxi, \mr{H} \ull{\ZZ} } \simeq 0$. This lift  may not be unique as its homotopy class can be varied by elements in the image of \[ 
\begin{tikzcd}
\updelta_*: \mr{H}^0(\mr{B}; \ull{\mathbb{F}}_2) \rar & \mr{H}^1(\mr{B}; \ull{\cal{A}}_{\mr{G}}^{\circ}),
\end{tikzcd}
\] 
where $\updelta:\mr{H}\ull{\mathbb{F}}_2 \longrightarrow \Sigma\mr{H}\ull{\cal{A}}_{\mr{G}}^{\circ}$ is the connecting map of the  fiber sequence in \eqref{fiberg} (the bottom horizontal row). In other words, the homotopy classes of all possible ${\sf f}^{{\sf g}}_{\upxi, \mr{H} \ull{\cal{A}}_{\mr{G}} }$ form a coset of the subgroup $\mr{D} = {\sf im}(\updelta_*)$ within $\mr{H}^1(\mr{B}; \ull{\cal{A}}_{\mr{G}}^{\circ})$. 

\begin{defn} \label{defn:ghosts} Let $\upxi$ be a $\mr{H}\ull{\cal{A}}_{\mr{G}}$-homogeneous  $\mr{G}$-equivariant vector bundle. Then a \emph{ghost} of the first Stiefel--Whitney class ${\sf w}_1^{\ull{\mathbb{Z}}}(\upxi)$ is the homotopy class of any map ${\sf f}^{{\sf g}}_{\upxi, \mr{H} \ull{\cal{A}}_{\mr{G}} }$ that fits in the diagram \eqref{fiberg}. Let 
\[ \mathfrak{G}_1(\upxi) := \{ [{\sf f}^{{\sf g}}_{\upxi, \mr{H} \ull{\cal{A}}_{\mr{G}} }] \in \mr{H}^1(\mr{B}; \ull{\cal{A}}_{\mr{G}}^{\circ}) : \text{diagram  \eqref{fiberg} is homotopy commutative}\} \]
 denote the $\mr{D}$-coset consisting of all possible ghosts of  the first Stiefel--Whitney class. 
\end{defn}

\bigskip
\begin{proof}[\textbf{Proof of} \Cref{thm:ghost}] It follows  from
  \eqref{fiberg} and  the uniqueness of the first Steifel-Whitney class (see \Cref{rmk:uniquew1}) that 
 \[ {\sf w}_1^{\ull{\cal{A}}_{\mr{G}} }(\upxi) := [ {\sf f}^{(0)}_{\upxi, \mr{H} \ull{\cal{A}}_{\mr{G}} }]= 0\] 
 whenever $0 \in \mathfrak{G}_1(\upxi)$.
\end{proof}

\bigskip
Next we will show that $0 \in \mathfrak{G}_1(\upxi)$ if $\mr{G}$ acts freely on the base space $\mr{B} $ of $\upxi$. 

\bigskip
\begin{proof}[\textbf{Proof of} \Cref{freeVSghost}]
Let $\upxi$ denote an $\mr{H}\ull{\mathcal{A}}_{\mr{G}}$-homogeneous bundle and let $\mr{B}$ denote the base space of $\upxi$. Since there is an $\EE_\infty^{\mr{G}}$-ring map 
\[ 
\begin{tikzcd}
\mr{H} \upiota \colon \mr{H} \ull{\mathcal{A}}_{\mr{G}} \ar[rr] && \mr{H} \ull{\mathbb{Z}}, 
\end{tikzcd}
\]
 an $\mr{H}\ull{\mathcal{A}}_{\mr{G}}$-orientation leads to an $\mr{H} \ull{\mathbb{Z}}$-orientation. We will now show that the converse is true when $\mr{G}$ acts freely on $\mr{B}$.

 As $\mr{B}$ has a free action,  we get a principal  $\mr{G}$-bundle
 $ 
 \mr{p} :\mr{B} \longrightarrow \mr{B}/\mr{G}
 $
 whose classifying map is 
 \[ 
 \begin{tikzcd}
 f:\mr{B} \rar &  \mathrm{EG}
 \end{tikzcd}
 \]
 on the total space. Notice that the composite 
 \[ 
 \begin{tikzcd}
 \mr{B} \rar["\Delta"] & \mr{B} \times \mr{B} \rar["1_{\mr{B}} \times f"] & \mr{B} \times \mr{EG} 
 \end{tikzcd}
 \]
 is a $\mr{G}$-equivalence. Therefore, 
 \begin{eqnarray*}
 \mr{H}^n_{\mr{G}}(\mr{B}; \ull{\cal{A}}_{\mr{G}}^{\circ})
   &\cong&
[\Sigma^{\infty}_{\mr{G}} \mr{B}_+, \Sigma^n\mr{H}\ull{\cal{A}}_{\mr{G}}^{\circ}]^{\mr{G}} \\
 &\cong&
[\Sigma^{\infty}_{\mr{G}}(\mr{B} \times \mathrm{EG})_+, \Sigma^n\mr{H}\ull{\cal{A}}_{\mr{G}}^{\circ}]^{\mr{G}} \\
& \cong&  [\Sigma^{\infty}_{\mr{G}}\mr{B}_+, \Sigma^n\mathrm{F}(\mathrm{EG}_+, \mr{H}\ull{\cal{A}}_{\mr{G}}^{\circ})]^{\mr{G}}
\end{eqnarray*}
for any natural number $n \in \mathbb{N}$. By setting  $\cal{F} = \{ 1\}$ and  $\mr{D} = \mr{EG}$ in \cite[II.2.2]{LMMS}, we conclude that the map
\[ 
\begin{tikzcd}
\mathrm{F}(\mathrm{EG}_+, \mr{H}\ull{\cal{A}}_{\mr{G}}) \ar[rr] & &
\mathrm{F}(\mathrm{EG}_+, \mr{H}\ull{\mathbb{Z}})
\end{tikzcd}
 \]
induced by $\mr{H}\upiota$ is a $\mr{G}$-equivalence. Thus, $\mathrm{F}(\mathrm{EG}_+, \mr{H}\ull{\cal{A}}_{\mr{G}}^{\circ})\simeq_{\mr{G}} *$.  Consequently, $\mr{H}^n(\mr{B}; \ull{\cal{A}}_{\mr{G}}^{\circ}) \cong 0$ for all $n \in \mathbb{N}$,  and in particular,   
$\mathfrak{G}_1(\upxi) = \{ 0\}$.  Thus, the result follows from \Cref{thm:ghost}. 
\end{proof}

\bigskip
\begin{rmk}
Any Atiyah Real bundle of a given dimension is  a homogeneous $\mr{C}_2$-equivariant bundle, and therefore it is also a $\mr{H}\ull{\mathcal{A}}_{\mr{C}_2}$-homogeneous bundle (see \Cref{rmk:homVSrelhom}). 
\end{rmk}

\bigskip
\begin{proof}[\textbf{Proof of} \Cref{thm:ARnotorientable}] Recall that the $\mr{C}_2$-fixed points of the tautological Atiyah Real bundle $\hat{\upgamma}$ is the tautological line bundle $\upgamma_1$ over  $\mathbb{RP}^\infty$ 
\[ \hat{\gamma}^{\mr{C}_2} \cong \upgamma_1 .\]
Consequently, we have a commutative diagram 
\[ 
\begin{tikzcd}
\mr{S}^1 \dar[hook, "\mr{V}_1"'] \rar[hook, "\mr{U}_1"] & \mathbb{RP}^{\infty} \dar[hook,"\mr{V}_2"'] \rar["\mr{U}_2"] &  \Omega^{\infty} \bo \dar[hook, "\mr{V}_3"]  \\ 
\hat{\mathbb{C}}\mathbb{P}^1 \rar[hook, "\mr{L}_1"] &  \hat{\mathbb{C}}\mathbb{P}^\infty \rar["\mr{L}_2"] & \Omega^{\infty}_{\mr{C}_2} {\sf br} , 
\end{tikzcd}
\]
where $\mr{V}_i$ are the inclusion of the fixed points, $\mr{U}_1$ and $\mr{L}_1$ are the skeletal inclusions, and,  $\mr{U}_2$ and $\mr{L}_2$ are the classifying maps for $\upgamma_1$ and $\hat{\upgamma}$ respectively. Note that 
\begin{enumerate}
\item  $\mr{U}_2 \circ \mr{U}_1$ is the classifying map for the m\"obius bundle and therefore represents the generator in $\pi_1\bo \cong \mathbb{Z}/2$, and, 
\item  $\mr{V}_{3\ast}: \pi_1\bo \longrightarrow \pi_1^{\mr{C}_2} {\sf br}$ is an isomorphism. 
\end{enumerate}
Consequently, the composition $\mr{L}_2 \circ \mr{L}_1 \circ \mr{V}_1$ represents the nonzero element \[ \eta \in \pi_1^{\mr{C}_2} {\sf br} \cong \mathbb{Z}/2.\] 
Greenlees \cite[pg 67]{GreenleeskoG} showed that \[ \pi_1^{\mr{C}_2}( \bo_{\mr{C}_2}) \cong \mr{RO}(\mr{C}_2)/\mr{RU}(\mr{C}_2) \cong \mathbb{Z}/2\lbrace [1], [\upsigma] \rbrace,\] and it is a standard fact that the natural map 
$\Omega^{\infty}_{\mr{C}_2} {\sf br} \longrightarrow \Omega^{\infty}_{\mr{C}_2} \bo_{\mr{C}_2}$
sends $\eta$ to $[1 + \upsigma]$ at the level of fundamental groups. Since the natural map from $ \bgl(\mathbb{S}_{\mr{C}_2})$ to $\bgl(\mr{H}\ull{\mathcal{A}}_{\mr{C}_2})$ is an isomorphism on $\pi_1$ and that the $\mr{J}$-homomorphism induces  an isomorphism $\ull{\pi}_1(\bo_{\mr{C}_2}) \cong \ull{\pi}_1( \bgl(\mathbb{S}_{\mr{C}_2}))$
of Mackey functors, it follows that the composition 
\[
\begin{tikzcd}
\mr{S}^1  \rar["\mr{L}_1 \circ \mr{V}_1"']   &  \hat{\mathbb{C}}\mathbb{P}^\infty
\rar["\mr{L}_2"'] \ar[rrrr, bend left, dashed, "{\sf f}^{(0)}_{\hat{\upgamma},
  \mr{H}\ull{\mathcal{A}}_{\mr{C}_2} }"'] & \Omega^{\infty}_{\mr{C}_2} {\sf br} \rar[] &
\Omega^{\infty}_{\mr{C}_2} \bo_{\mr{C}_2} \rar[" \mJJ^{\circ}_{\mr{C}_2}"'] &
\bgl(\mathbb{S}_{\mr{C}_2}) \rar[""']  & \bgl(\mr{H}\ull{\mathcal{A}}_{\mr{C}_2}), 
\end{tikzcd}
\]
is not null. Thus, ${\sf w}_1^{\ull{\mathcal{A}}_{\mr{C}_2}}(\hat{\upgamma}) := [{\sf f}^{(0)}_{\hat{\upgamma}, \mr{H}\ull{\mathcal{A}}_{\mr{C}_2} }]$ cannot equal zero. 
\end{proof}
\subsection{Nonexistence of $\mr{H}\ull{\mathcal{A}}_{\mr{G}}$-orientation of $\upgamma_{\uprho}$} \ 

The underlying nonequivariant bundle of $\upgamma_{\uprho}$ is isomorphic to  $
\upgamma_1^{\oplus |\mr{G}|}$, the
$|\mr{G}|$-fold direct sum of the tautological line bundle $\upgamma_1$ over
$\mathbb{RP}^{\infty}$. Therefore, the restriction map 
\[ 
\begin{tikzcd}
{\sf res}^{\mr{G}}_{\sf e}: \mr{H}^{1}_{\mr{G}}(\mr{B}_{\mr{G}} \Sigma_2; \ull{\mathbb{F}}_2) \ar[rr] && {\mr{H}^{1}(\mathbb{RP}^\infty; \mathbb{F}_2) }
\end{tikzcd}
\]
sends the first  $\mr{G}$-equivariant Stiefel--Whitney class ${\sf w}_1^{\ull{\mathbb{Z}}}(\upgamma_{\uprho})$ to  
\[
 {\sf  w}_1^{\mathbb{Z}}(\upgamma_1^{\oplus |\mr{G}|})   = |\mr{G}|\cdot {\sf  w}_1^{\mathbb{Z}}(\upgamma_1) ,\]  
 which is nonzero when $|\mr{G}|$ is odd. Thus,  ${\sf w}_1^{\ull{\mathbb{Z}}}(\upgamma_{\uprho})$  must also be nonzero in this case, and by \Cref{thm:w1orient}, we conclude that $\upgamma_{\uprho}$ is not $\mr{H} \ull{\mathbb{Z}}$-orientable. 

\begin{notn} For a group $\mr{G}$ and its subgroup $\mr{H}$, let
  \[ 
  \begin{tikzcd}
  \upiota_{\mr{H}}: \Top^{\mr{G}} \rar & \Top^{\mr{H}}
  \end{tikzcd}
     \]
 denote the functor that restricts the action of $\mr{G}$ to $\mr{H}$. The right adjoint of $\upiota_{\mr{H}}$ is the functor $\mr{Map}_{\mr{H}}(\mr{G}, -)$ which converts an $\mr{H}$-space $\mr{X}$  to a $\mr{G}$-space  consisting of $\mr{H}$-equivariant maps from $\mr{G}$ to $\mr{X}$. Let 
 \[ 
 \begin{tikzcd}
\upeta: \mathbbm{1}_{\Top^{\mr{G}}} \rar & \mr{Map}_{\mr{H}}(\mr{G}, \upiota_{\mr{H}}(-))
 \end{tikzcd}
 \]
 denote the unit of this adjunction. We also use $\upiota_{\mr{H}}$ to denote the corresponding restriction functor $\Sp^{\mr{G}}$ to $\Sp^{\mr{H}}$.
 \end{notn}
  \begin{notn} \label{Indbundle}
 Let $\mr{B}$ be  a $\mr{G}$-space   such that its restriction
 $\upiota_{\mr{H}} (\mr{B})$ is the base space of the  $\mr{H}$-equivariant
 bundle $\upxi$. Then $\mr{Map}_{\mr{H}}(\mr{G}, \upxi)$  is a
 $\mr{G}$-equivariant bundle over  $\mr{Map}_{\mr{H}}(\mr{G},
 \upiota_{\mr{H}}(\mr{B}))$. 
 We let 
 \[ {\sf Ind}_{\mr{H}}^{\mr{G}}(\upxi) := \upeta_{\mr{B}}^{\ast} \mr{Map}_{\mr{H}}(\mr{G}, \upxi) \] 
 denote the $\mr{G}$-equivariant bundle over $\mr{B}$ obtaining by 
 pulling back $\mr{Map}_{\mr{H}}(\mr{G}, \upxi)$ along the natural map $\upeta_{\mr{B}}: \mr{B} \longrightarrow \mr{Map}_{\mr{H}}(\mr{G}, \upiota_{\mr{H}}(\mr{B}))$. 
\end{notn}
  \begin{thm}  \label{SWcompare} Let $\R$ be an $\mathbb{E}^{\mr{G}}_{\infty}$-ring spectrum. Suppose $\upxi$ is an 
    $\mr{H}$-equivariant $\upiota_{\mr{H}}(\mr{R})$-homogeneous bundle over  $\upiota_{\mr{H}} (\mr{B})$ where $\mr{B}$ is a $\mr{G}$-space. Then 
  \[ {\sf w}_1^{\mr{R}}({\sf Ind}_{\mr{H}}^{\mr{G}}(\upxi)) = {\sf tr} ({\sf w}_1^{\upiota_\mr{H}(\mr{R})}(\upxi)), \] 
 where ${\sf tr}: [\upiota_{\mr{H}}(\mr{B})_+, \bgl(\upiota_{\mr{H}}(\mr{R})) ]^{\mr{H}} \longrightarrow [\mr{B}_+, \bgl(\mr{R})]^{\mr{G}} $ is the transfer map. 
 \end{thm}
 \begin{proof}[\textbf{Proof}] The classifying map of $\mr{Map}_{\mr{H}}(\mr{G},\upxi)$
   corresponds to that of $\upxi$ under the isomorphism   
\begin{eqnarray*} 
[ \mr{Map}_{\mr{H}}(\mr{G}, \upiota_{\mr{H}}(\mr{B}))_+, \Omega_{\mr{G}}^{\infty} \ko_{\mr{G}}]^{\mr{G}} &\cong& [\Sigma^{\infty}_{\mr{G}} \mr{Map}_{\mr{H}}(\mr{G}, \upiota_{\mr{H}}(\mr{B}))_+,  \ko_{\mr{G}}]^{\mr{G}} \\
&\cong& [ \mr{Map}_{\mr{H}}(\mr{G}_+, \Sigma^{\infty}_{\mr{H}}\upiota_{\mr{H}}(\mr{B}_+)) , \ko_{\mr{G}}]^{\mr{G}} \\
&\cong & [  \Sigma^{\infty}_{\mr{H}} \upiota_{\mr{H}}(\mr{B})_+ , \ko_{\mr{H}}]^{\mr{H}}  \\
&\cong & [\upiota_{\mr{H}}(\mr{B})_+, \Omega_{\mr{H}}^{\infty} \ko_{\mr{H}}]^{\mr{H}},
\end{eqnarray*}
and its precomposition  with 
$\upeta_{\mr{B}}: \mr{B}_+ \to \mr{Map}_{\mr{H}}(\mr{G},
\upiota_{\mr{H}}(\mr{B}))_+$ is the classifying map of ${\sf Ind}_{\mr{H}}^{\mr{G}}(\upxi)$.
However, the isomorphism above composed with  $\upeta_{\mr{B}}^{\ast}$ is   simply the transfer map  
\[ 
\begin{tikzcd}
  {\sf tr} : [ \upiota_{\mr{H}}(\mr{B})_+ , \ko_{\mr{H}}]^{\mr{H}} \rar
  &  {[ \mr{B}_+, \ko_{\mr{G}} ]^{\mr{G}}, }
\end{tikzcd}
\] 
 of the Mackey functor $\ull{\ko}_{\mr{G}}^{0} (\mr{B}_+ )$, and 
 the commutative diagram 
\[ 
\begin{tikzcd}
{[  \upiota_{\mr{H}}(\mr{B})_+ , \ko_{\mr{H}}]^{\mr{H}}} \rar["{\sf tr}"] \dar[" \pic(\iota_{\upiota_{\mr{H}}(\mr{R})} ) \circ \mJJ_{\mr{H} \ast}"'] &  {[ \mr{B}_+, \ko_{\mr{G}} ]^{\mr{G}}, } \dar[" \pic(\iota_{\mr{R}}) \circ \mJJ_{\mr{G} \ast} "] \\
{[  \upiota_{\mr{H}}(\mr{B})_+ , \pic(\upiota_{\mr{H}}(\mr{R}) )]^{\mr{H}}} \rar["{\sf tr}"] &  {[ \mr{B}_+, \pic(\mr{R}) ]^{\mr{G}}, }
\end{tikzcd}
\] 
implies $[{\sf f}_{{\sf Ind}_{\mr{H}}^{\mr{G}}(\upxi), \mr{R}}] = {\sf tr}([{\sf f}_{\upxi, \upiota_{\mr{H}}(\mr{R}) } ])$.  Consequently, ${\sf Ind}_{\mr{H}}^{\mr{G}}(\upxi)$ is $\mr{R}$-homogeneous when $\upxi$ is $\upiota_{\mr{H}}(\mr{R})$-homogeneous and ${\sf w}_1^{\mr{R}}({\sf Ind}_{\mr{H}}^{\mr{G}}(\upxi)) = {\sf tr} ({\sf w}_1^{\upiota_\mr{H}(\mr{R})}(\upxi))$. 
\end{proof}

\bigskip
\begin{proof}[\textbf{Proof of} \Cref{thm:gamma_rho_orientation} ] 
As noted after \eqref{C2bundleMap},  when $\mr{G}$ is the group of order
  $2$, $\upgamma_{\uprho_{\mr{G}}}$ is $\mr{H}\ull{\mathbb{Z}}$-orientable. 
Now let $\mr{G}$ be an  arbitrary group of even order. A simple counting
argument shows that there exists an element (other than $1$) which is its own
inverse. Thus,  $\mr{G}$  admits at least one subgroup of order $2$;  choose any one of them and denote it by $\mr{T}$. Observe that the bundle $\upgamma_{\uprho_{\mr{G}}}$ is isomorphic to  ${\sf Ind}_{\mr{T}}^{\mr{G}}( \upgamma_{\uprho_{\mr{T}}})$ (see \Cref{Indbundle}).
Since $\upgamma_{\mr{T}}$ is $\mr{H}\ull{\mathbb{Z}}$-orientable, it follows from 
\Cref{SWcompare} that 
\[ {\sf w}_1^{\ull{\mathbb{Z}}}(\upgamma_{\uprho_{\mr{G}}}) = {\tr} ({\sf w}_1^{\ull{\mathbb{Z}}}(\upgamma_{\uprho_{\mr{T}}}) ) ={\sf tr}( 0) = 0,   \]
thus $\upgamma_{\uprho_{\mr{G}}}$ is $\mr{H}\ull{\mathbb{Z}}$-orientable. 

Now we  show that $\upgamma_{\uprho_{\mr{G}}}$ is not $\mr{H}
\ull{\cal{A}}_{\mr{G}}$-orientable. By construction
$\upgamma_{\uprho_{\mr{G}}}$ is homogeneous, therefore $\mr{H}\ull{\cal{A}}_{\mr{G}}$-homogeneous
(see \Cref{rmk:homVSrelhom}). By \Cref{thm:fixBGPi}, each of the components of
$\mr{G}$-fixed points of $\mr{B}_{\mr{G}}\Sigma_2$  is isomorphic to
$\mr{B}\Sigma_2$. Note that, $\upgamma_{\uprho_{\mr{G}}}$ is isomorphic to  ${\sf  Ind}_{\sf e}^{\mr{G}}(\upgamma_1)  $. Furthermore, if we restrict $\upgamma_{\uprho_{\mr{G}}}$ along a nontrivial loop in any component of  $(\mr{B}_{\mr{G}}\Sigma_2)^{\mr{G}}$,  we 
get ${\sf Ind}_{\sf e}^{\mr{G}}(\mathfrak{m})$
\[
\begin{tikzcd}
{\sf Ind}_{\sf e}^{\mr{G}}(\mathfrak{m}) \ar[rrr, dashed] \dar[dashed] &&& {\sf Ind}_{\sf e}^{\mr{G}}(\upgamma_{1}) \dar[] \\
 \mr{S}^1  \rar[hook] & \mr{B}\Sigma_2   \rar[hook] & (\mr{B}_{\mr{G}}\Sigma_2)^{\mr{G}} \rar[hook] &  \mr{B}_{\mr{G}}\Sigma_2, 
   \end{tikzcd}
  \]

 where $\mathfrak{m}$ is the mobius bundle over $\mr{S}^1$. Thus, in order to show that $\upgamma_{\uprho_{\mr{G}}}$ is not $\mr{H}\ull{\cal{A}}_{\mr{G}}$-orientable, it is enough to show it for ${\sf Ind}_{\sf e}^{\mr{G}}(\mathfrak{m})$. 
  
 It is a classical fact (see Axiom 4 \cite[$\mathsection$4]{MSchar}) that  ${\sf w}_1^{\mathbb{Z}}(\mathfrak{m})$ is the element 
 \[-1 \in \mathbb{Z}^{\times} \cong \mr{H}^{1} ( \mr{S}^1; \mathbb{F}_2) \cong \ull{\cal{A}}_{\mr{G}}^{\times}(\mr{G}/{\sf e}).\] 
A  standard formula for the norm  of the sum of two elements (see \cite[Example~1.6.3]{KMazur}) implies that 
   \[ \mathrm{norm}(-1) \equiv -1 \mod \langle [\mr{G}/\mr{H}]: \text{$\mr{H} \subsetneq \mr{G}$}\rangle \]
 in $\ull{\cal{A}}_{\mr{G}}(\mr{G}/\mr{G})$.  In particular, $\mathrm{norm}(-1) \neq 1$, thus 
 \[ \tr(-1) \in \mr{H}_{\mr{G}}^1(\mr{S}^1 ; \ull{\cal{A}}_{\mr{G}}^{\times}) \cong \ull{\cal{A}}_{\mr{G}}^{\times}(\mr{G}/\mr{G})  \]
  is a nontrivial element.
By \Cref{SWcompare} the first Stiefel--Whitney class for ${\sf Ind}_{\sf e}^{\mr{G}}(\mathfrak{m})$ 
  \[ {\sf w}^{\ull{\cal{A}}_{\mr{G}}}_1({\sf Ind}_{\sf e}^{\mr{G}}(\mathfrak{m})) = {\sf tr}(-1) \in \mr{H}^1(\mr{B}; \ull{\cal{A}}^{\times}_{\mr{G}}) \]
  is also nonzero. Consequently  ${\sf Ind}_{\sf e}^{\mr{G}}(\mathfrak{m})$, and therefore $\upgamma_{\uprho_{\mr{G}}}$, does not admit an $\mr{H}\ull{\cal{A}}_{\mr{G}}$-orientation. 
  \end{proof}
  
\appendix
\section{Units of Burnside Tambara functors} \label{Appendix:unit}
In anticipation of future applications of \Cref{thm:ghost}, we record the Mackey functor structure of  $\ull{\mathcal{A}}_{\mr{G}}^{\times}$ for a few familiar  abelian $2$-groups.
 The starting point of  these calculations is a fundamental result due to Matsuda \cite{Matsuda} which identifies $ \ull{\mathcal{A}}_{\mr{G}}^{\times}(\mr{G}/\mr{H}) := \mr{A}_{\mr{H}}^{\times}$, where $\mr{A}_{\mr{H}}$ is the Burnside ring of the group $\mr{H}$.  
\begin{thm}[Matsuda] \label{thm:Matsuda} For a finite group $\mr{G}$, the units $\mr{A}_{\mr{G}}^{\times} \cong (\mathbb{Z}/2)^{m +1}$,  where $m$ is the cardinality of the set 
\[ \mathcal{I} =  \{ \mr{H}: \text{$\mr{H} \subset \mr{G}$ has index $2$} \}/(\text{conjugation}).\]
The elements  $[\mr{G}/\mr{H}] - 1$ for $\mr{H} \in \mathcal{I}$, along with $-1$, form a basis for $\mr{A}_{\mr{G}}^{\times}$.
\end{thm}
While the restriction maps of $\ull{\mathcal{A}}_{\mr{G}}^{\times}$  is straightforward to determine, the  transfer maps are deduced from  the formula of  norm map $\mr{N}_{\mr{H}}^{\mr{G}}: \ull{\mathcal{A}}_{\mr{G}}(\mr{G}/\mr{H}) \longrightarrow  \ull{\mathcal{A}}_{\mr{G}}(\mr{G}/\mr{G})
$ given by 
\begin{equation} \label{norm}
\begin{tikzcd}
\mr{N}_{\mr{H}}^{\mr{G}}(\mr{X}) =  \mr{Map}_{\mr{H}}(\mr{G}, \mr{X}), 
\end{tikzcd}
 \end{equation}
where $\mr{X}$ is   a finite $\mr{H}$-set (see \cite[Example 1.4.6]{KMazur} as well as \cite[Appendix A]{Nakaoka}).  The Mackey functor structure of $\ull{\cal{A}}_{\mr{G}}^{\times}$  can be complicated depending on the structure of $\mr{G}$ and 
  has been the subject of research for several decades \cite{Matsuda,MatsudaMiyata,Yoshida, Yalcin, Bouc1, BarkerL, Bouc2,Valero}.

When $\mr{G} = \mr{C}_{2^n}$, then the subgroups of $\mr{G}$ are cyclic of order $2^k$, in fact, there is exactly one for each $0 \leq k \leq n$. We will denote this subgroup by $\mr{C}_{2^k}$. From \Cref{thm:Matsuda}, we get 
 \[ \ull{\mathcal{A}}_{\mr{C}_{2^n}}^{\times}(\mr{C}_{2^n}/\mr{C}_{2^k}) \cong \mathbb{Z}/2\{ {\sf a}_k, {\sf b}_k \}, \]
 where ${\sf a}_k = -1$ and ${\sf b}_k = [\mr{C}_{2^k}/\mr{C}_{2^{k-1}}] -1$. 
 \begin{prop} \label{prop:cyclic} The restriction and the transfer map in the diagram 
 \[ 
  \begin{tikzcd}
  \ull{\mathcal{A}}_{\mr{C}_{2^n}}^{\times}(\mr{C}_{2^n}/\mr{C}_{2^k}) \ar[dd, bend right, "{\sf res}_k"'] \\
  \\
  \ull{\mathcal{A}}_{\mr{C}_{2^n}}^{\times}(\mr{C}_{2^n}/\mr{C}_{2^{k-1}}) \ar[uu, bend right, "{\sf tr}_k"']
  \end{tikzcd}
 \]
 satisfies
 $ {\sf tr}_k({\sf a}_{k-1}) = {\sf tr}_k({\sf b}_{k-1}) = {\sf b}_{k}, \  {\sf res}_k({\sf a}_k) = {\sf a}_{k-1}, \  {\sf res}_k({\sf b}_k) = 0. $
\end{prop}
\begin{proof}[\textbf{Proof}] 
 The map ${\sf res}_k$ is immediate from restricting the actions. The transfer is  obtained by calculating the norm in the Burnside Tambara functor $\ull{\mathcal{A}}_{G}$ using the formula \eqref{norm}. 
 When $\mr{H} = \mr{C}_{2^{k-i}}$ and $\mr{G}=\mr{C}_{2^k}$ in \eqref{norm}, we get
\begin{equation*} 
\begin{array}{clll}
\mathrm{N}_{\mr{C}_{2^{k-i}}}^{\mr{C}_{2^k}} (l) &=& l + \frac{l^2-l}{2}[\mr{C}_{2^k}/\mr{C}_{2^{k-1}}] +
\frac{l^4 - l^2}{4} [\mr{C}_{2^k}/\mr{C}_{2^{k-2}}] + \dots   \\ 
&& \hspace{150pt}  \cdots  + \frac{l^{2^{k-i}} - l^{2^{k-i-1}}}{2^{i}} [\mr{C}_{2^k}/ \mr{C}_{2^{k-i}} ].
\end{array}
\end{equation*}
By setting $i= 1$ and $l =-1$ in the above equation, we get $ {\sf tr}_k({\sf a}_{k-1}) =  {\sf b}_{k}$, and by setting $i=k$ and $l =-1$,   we get $ {\sf tr}_e^{\mr{C}_{2^k}}({\sf a}_0) = {\sf b}_k$ for all $0 \leq k \leq n$.
Thus, 
\begin{eqnarray*}
{\sf tr}_k({\sf b}_{k-1}) & =& {\sf tr}_k ( {\sf tr}_e^{\mr{C}_{2^{k-1}}}({\sf a}_0)) \\
&=&  {\sf tr}_e^{\mr{C}_{2^k}}({\sf a}_0) \\
&=& {\sf b}_k 
\end{eqnarray*}
as desired. 
\end{proof}

\bigskip
When $ \mr{G} = \mr{C}_2 \times \mr{C}_2,$ there  are three different order $2$ subgroups $\mr{G}_{01} := \mr{C}_2 \times \{ 0\} $, $\mr{G}_{10} = \{ 0\} \times \mr{C}_2$ and $\mr{G}_{11} = \{(x,x): x \in \mr{C}_2\}$. Using \Cref{thm:Matsuda}, we note: 
\begin{align*}
 \ull{\mathcal{A}}_{\mr{G}}^{\times}(\mr{G}/\mr{G})
  &= \mathbb{Z}/2\{{\sf a}_{2}, {\sf b}_{2,10}, {\sf b}_{2,01}, {\sf b}_{2,11}\}
   \\
   \ull{\mathcal{A}}_{\mr{G}}^{\times}(\mr{G}/\mr{G}_{10})
  &= \mathbb{Z}/2\{{\sf a}_{10}, {\sf b}_{10}\}\\
  \ull{\mathcal{A}}_{\mr{G}}^{\times}(\mr{G}/\mr{G}_{01})
  &= \mathbb{Z}/2\{{\sf a}_{01}, {\sf b}_{01}\}\\
  \ull{\mathcal{A}}_{\mr{G}}^{\times}(\mr{G}/\mr{G}_{11})
  &= \mathbb{Z}/2\{{\sf a}_{11}, {\sf b}_{11}\}\\
 \ull{\mathcal{A}}_{\mr{G}}^{\times}(\mr{G}/{\sf e})
  &= \mathbb{Z}/2\{{\sf a}_{0}\},
\end{align*}
where ${\sf a}_0, {\sf a}_{10}, {\sf a}_{01}, {\sf a}_{11}$ and  ${\sf a}_{2}$  denote $-1$ in their respective rings, ${\sf b}_{ij}$ denotes $[\mr{G}_{ij}] -1$, and, ${\sf b}_{2, ij}$ denotes $[\mr{G}/ \mr{G}_{ij}] - 1$.
\begin{prop} \label{C2C2} The restriction and transfer maps as indicated in the diagrams 
\[
\begin{tikzcd}
\ull{\mathcal{A}}_{\mr{G}}^{\times}(\mr{G}/\mr{G}) \ar[dd, bend right, "{\sf res}_{2, ij}"']  &&  \ull{\mathcal{A}}_{\mr{G}}^{\times}(\mr{G}/\mr{G}_{ij}) \ar[dd, bend right, "{\sf res}_{ij}"']  \\  
\\
\ull{\mathcal{A}}_{\mr{G}}^{\times}(\mr{G}/\mr{G}_{ij}) \ar[uu, bend right, "{\sf tr}_{2, ij}"'] && \ull{\mathcal{A}}_{\mr{G}}^{\times}(\mr{G}/{\sf e})  \ar[uu, bend right, "{\sf tr}_{ij}"']
\end{tikzcd}
\]
is determined by the formulas ${\sf res}_{2, ij}({\sf a}_2) = {\sf a}_{ij}, \ {\sf res}_{ij}({\sf a}_{ij}) = {\sf a}_0$, 
\[ 
{\sf res}_{2, ij}({\sf b}_{2,kl}) = \left\lbrace \begin{array}{ccc}
0 & \text{if $ij = kl$,} \\
{\sf b}_{ij}  &\text{otherwise}, 
\end{array} \right.
\]
 $ {\sf tr}_{ij}({\sf a}_{0}) = {\sf b}_{ij} $, ${\sf tr}_{ij}({\sf a}_{ij}) ={\sf b}_{ij} $ and  $ {\sf tr}_{2, ij}({\sf b}_{ij}) = {\sf b}_{2,10} + {\sf b}_{2,01} + {\sf b}_{2,11} $. 
\end{prop}
\begin{proof}[\textbf{Proof}] All claims, except for the formula for ${\sf tr}_{2, ij}({\sf b}_{ij})$,   can be verified directly using arguments very similar to that in the proof of \Cref{prop:cyclic}.  
  
  To see the formula for ${\sf tr}_{2, ij}({\sf b}_{ij})$, we first observe from \eqref{norm} that 
\begin{equation*}
\mr{N}_{\sf e}^{\mr{G}} (-1) = -1 + [\mr{G}/ \mr{G}_{01}] + [\mr{G}/\mr{G}_{10}] + [\mr{G}/\mr{G}_{11}] - [\mr{G}/{\sf e}].
\end{equation*}
If  ${\sf tr}_{2, ij}({\sf b}_{ij}) =  \epsilon_1{\sf a}_2 + \epsilon_2 {\sf  b}_{2,10} +\epsilon_3 {\sf b}_{2,01}+ \epsilon_4{\sf b}_{2,11}$, it means
\begin{equation*}
\mr{N}_{\sf e}^{\mr{G}} (-1) = (-1)^{\epsilon_1}([\mr{G}/ \mr{G}_{01}]-1)^{\epsilon_2}([\mr{G}/ \mr{G}_{10}]-1)^{\epsilon_3}([\mr{G}/ \mr{G}_{11}]-1)^{\epsilon_4}.
\end{equation*}
Comparing the two equations for $\mr{N}_{\sf e}^{\mr{G}} (-1)$, we conclude  $\epsilon_2 = \epsilon_3 = \epsilon_4 = 1$ and $\epsilon_1 =0$.
\end{proof}
\Cref{C2C2} can be readily extended to $\mathbb{T}_n :=  \mr{C}_2^{\times n}$. Any subgroup of $\mathbb{T}_n$ is isomorphic to $\mathbb{T}_k$ for some $k \in \{ 0,1, \dots, n \}$, and there are exactly \[ \frac{(2^n -1)\dots(2^n - 2^{k-1})}{(2^k-1) \dots (2^{k} - 2^{k-1})}\] of them.  
Fix a subgroup $\mr{H}$ isomorphic to $\mathbb{T}_k$. The above formula implies there are exactly $2^k -1$ index $2$ subgroups of $\mr{H}$. Therefore, by \Cref{thm:Matsuda}, 
\[ \ull{\cal{A}}^{\times}_{\mathbb{T}_n}(\mathbb{T}_n/\mr{H}) \cong (\mathbb{Z}/2)^{\times 2^k}\]
 generated by ${\sf a}_{\mr{H}} = -1$ and ${\sf b}_{\mr{H}/\mr{K}'} = [\mr{H}/\mr{K}'] -1 $ where $\mr{K}'$ is an index $2$ subgroup of $\mr{H}$.  Likewise, for a given index $2$ subgroup $\mr{K}$ of $\mr{H}$
\[ \ull{\cal{A}}^{\times}_{\mathbb{T}_n}(\mathbb{T}_n/\mr{K}) \cong (\mathbb{Z}/2)^{\times 2^{k-1}}\] generated by ${\sf a}_{\mr{K}} = -1$ and ${\sf b}_{\mr{K}/\mr{L}'} -1 $ where $\mr{L}'$ is an index $2$ subgroup of $\mr{K}$. 

It is easy to see that the restriction map 
\[ 
\begin{tikzcd}
{\sf res}_{\mr{K}}^{\mr{H}}: \ull{\cal{A}}^{\times}_{\mathbb{T}_n}(\mathbb{T}_n/\mr{H}) \rar & \ull{\cal{A}}^{\times}_{\mathbb{T}_n}(\mathbb{T}_n/\mr{K}) 
\end{tikzcd}
\]
sends ${\sf a}_{\mr{H}}$ to ${\sf a}_{\mr{K}}$, ${\sf b}_{\mr{H}/\mr{K}'}$ to ${\sf b}_{\mr{K}/ (\mr{K} \cap \mr{K}')}$ if $\mr{K}' \neq \mr{K}$, and ${\sf b}_{\mr{H}/\mr{K}} $ to $0$. The transfer map 
\[ 
\begin{tikzcd}
{\sf tr}_{\mr{K}}^{\mr{H}}: \ull{\cal{A}}^{\times}_{\mathbb{T}_n}(\mathbb{T}_n/\mr{K}) \rar & \ull{\cal{A}}^{\times}_{\mathbb{T}_n}(\mathbb{T}_n/\mr{H}) 
\end{tikzcd}
\]
sends ${\sf a}_{\mr{K}}$ to ${\sf b}_{\mr{H}/\mr{K}}$, and ${\sf b}_{\mr{K}/\mr{L}'}$ to ${\sf b}_{\mr{H}/\mr{K}} +{\sf b}_{\mr{H}/\mr{K}_1} + {\sf b}_{\mr{H}/\mr{K}_2}$, where $\mr{K}_1$ and $\mr{K}_2$ are the two index $2$ subgroups of $\mr{H}$ whose intersection with $\mr{K}$ equals $\mr{L}'$.

\bibliographystyle{amsalpha}
\bibliography{Equiv-orientation}
\end{document}